\documentclass{article}

\usepackage[english]{babel}

\usepackage[letterpaper,top=2cm,bottom=2cm,left=3cm,right=3cm,marginparwidth=1.75cm]{geometry}
\usepackage[toc,page]{appendix}
\usepackage[algo2e,linesnumbered,vlined,ruled]{algorithm2e}
\usepackage{amsmath,amssymb,amsthm,color}
\usepackage{graphicx}
\usepackage[colorlinks=true, allcolors=blue]{hyperref}

\newtheorem{theorem}{Theorem}[section]
\newtheorem{lemma}[theorem]{Lemma}
\newtheorem{corollary}[theorem]{Corollary}

\newtheorem{definition}[theorem]{Definition}
\newtheorem{assumption}[theorem]{Assumption}
\newtheorem{observation}[theorem]{\textbf{Observation}}

\newtheorem{remark}[theorem]{Remark}

\usepackage{algorithm}
\usepackage{algpseudocode}
\usepackage{verbatim}

\newcommand{\hpk}{\hat{\Phi}_k}
\newcommand{\txkm}{\tilde{x}_{k,t-1}}
\newcommand{\txk}{\tilde{x}_{k,t}}
\newcommand{\cX}{\mathcal{X}}
\newcommand{\cE}{\chi}
\newcommand{\RR}{\mathbb{R}}
\usepackage{xcolor}

\usepackage{enumitem}
\usepackage{amsopn}

\usepackage{subcaption}
\newcommand{\email}[1]{\texttt{#1}}
\usepackage[numbers]{natbib} 

\title{Efficient First Order Method for Saddle Point Problems with Higher Order Smoothness\thanks{The first two authors equally contribute to this paper.}}
\author{Nuozhou Wang\thanks{Department of Industrial \& Systems Engineering, University of Minnesota
  (\email{wang9886@umn.edu}).} %\url{http://www.imag.com/\string~ddoe/}).}
\and Junyu Zhang\thanks{Department of Industrial Systems Engineering and Management, National University of Singapore
  (\email{junyuz@nus.edu.sg}). }
\and Shuzhong Zhang\thanks{Department of Industrial \& Systems Engineering, University of Minnesota
  (\email{zhangs@umn.edu}).}}

\begin{document}

\maketitle

\begin{abstract}
This paper studies the complexity of finding approximate stationary points for the smooth nonconvex-strongly-concave (NC-SC) saddle point problem: $\min_x\max_yf(x,y)$. Under the standard first-order smoothness conditions where $f$ is $\ell$-smooth in both arguments and $\mu_y$-strongly concave in $y$, existing literature shows that the optimal complexity for first-order methods to obtain an $\epsilon$-stationary point is $\tilde{O}\big(\sqrt{\kappa_y}\ell\epsilon^{-2}\big)$, where $\kappa_y=\ell/\mu_y$ is the condition number. However, when $\Phi(x):=\max_y f(x,y)$ has $L_2$-Lipschitz continuous Hessian in addition, we derive a first-order algorithm with an $\tilde{O}\big(\sqrt{\kappa_y}\ell^{1/2}L_2^{1/4}\epsilon^{-7/4}\big)$ complexity by designing an accelerated proximal point algorithm enhanced with the “Convex Until Proven Guilty” technique. Moreover, an improved $\Omega\big(\sqrt{\kappa_y}\ell^{3/7}L_2^{2/7}\epsilon^{-12/7}\big)$ lower bound for first-order method is also derived for sufficiently small $\epsilon$. As a result, given the second-order smoothness of the problem, the complexity of our method improves the state-of-the-art result by a factor of $\tilde{O}\big(\big(\frac{\ell^2}{L_2\epsilon}\big)^{1/4}\big)$, while almost matching the lower bound except for a small $\tilde{O}\big(\big(\frac{\ell^2}{L_2\epsilon}\big)^{1/28}\big)$ factor.
\end{abstract}

\section{Introduction}
\label{sec:intro}
In this work, we consider the smooth nonconvex-strongly-concave (NC-SC)  saddle point problem 
\begin{equation}
    \label{prob:minimax}
    \min _{x\in\mathbb{R}^{n}} \max _{y\in\mathbb{R}^{n'}} f(x, y),
\end{equation} 
where $f(x, y)$ is $\ell$-smooth, $\mu_y$-strongly concave in $y$, but possibly nonconvex in $x$. Such problems are widely encountered in generative model \cite{goodfellow2020generative}, robustness in adversarial training \cite{sinha2017certifying}, distributionally robust regression \cite{shafieezadeh2015distributionally}, classification with fairness \cite{nouiehed2019solving}, and so on. For NC-SC problems, the common goal is to find  
an $\epsilon$-stationary point $\bar x$ for the primal problem, namely, 
\(\|\nabla \Phi(\bar x)\|\leq \epsilon\)
with $\Phi(x):=\max_y f(x,y)$. To obtain an $\epsilon$-stationary point, \cite{lin2020gradient} proposed a two-time-scale gradient method with $O(\kappa_y^2\ell\epsilon^{-2})$ complexity, where $\kappa_y=\ell/\mu_y$ is the condition number of $f(x,\cdot)$ and the complexity is counted by the total number of gradient and function evaluations of the algorithm. By applying the inexact Proximal Point Algorithm (PPA) to solve the NC-SC problem while using Nesterov's acceleration to solve the subproblems, \cite{lin2020near} obtained an improved complexity of $\tilde{O}\big(\sqrt{\kappa_y}\ell\epsilon^{-2}\big)$, the same complexity is also achieved by \cite{zhang2021complexity} via the Catalyst scheme. In terms of the lower bound, under the standard first-order smoothness conditions, \cite{li2021complexity,zhang2021complexity} concurrently extended the hard instances in \cite{carmon2020lower,carmon2021lower} from nonconvex minimization to the NC-SC minimax setting, showing that $\Omega\big(\sqrt{\kappa_y}\ell\epsilon^{-2}\big)$ first-order oracles are required to find $\epsilon$-stationary points, which indicates the optimality of both \cite{lin2020near,zhang2021complexity}. 

Although the above results are tight for NC-SC problems that possess only gradient Lipschitz continuity, further improvements can still be expected for problems with higher-order smoothness. In \cite{luo2021finding}, a variant of the cubic regularized Newton's method was proposed to solve the NC-SC saddle point problem. To find an $\epsilon$-stationary point, \cite{luo2021finding} requires $\tilde{O}\big(\kappa_y^{1.5}\sqrt{\ell_2}\epsilon^{-1.5}\big)$ Hessian computations and $\tilde{O}\big(\kappa_y^2\sqrt{\ell_2}\epsilon^{-1.5}\big)$ gradient evaluations. Regardless of the Hessian evaluation and the complexity of solving cubic subproblems, \cite{luo2021finding} improves $\epsilon$-dependence of the first-order oracle complexity by a factor of $O(\epsilon^{-0.5})$.
Besides the research under the NC-SC setting, the general $p$-th order method \cite{bullins2022higher,lin2023monotone,lin2022perseus,jiang2022generalized} and the Quasi-Newton method \cite{asl2023j} have been considered for the monotone inequalities or convex-concave (C-C) setting, and \cite{chen2022cubic} considered the nonconvex-nonconcave (NC-NC) settings.

Despite the improved first-order oracle complexity, the expensive computation of higher-order oracles prevents the applications of higher-order methods in large-scale problems. Therefore, researchers have put significant efforts to designing first-order methods that can exploit high-order smoothness without requiring higher-order oracles. If we ignore the inner maximization structure of $\Phi(x)$  and consider merely the nonconvex optimization problem: $\min_{x} \Phi(x)$, several papers have achieved this purpose \cite{agarwal2016finding,carmon2018accelerated,carmon2017convex} with $\tilde{O}(\epsilon^{-7/4})$ complexity for first-order methods. Hence, a straightforward attempt to exploit the higher-order smoothness of NC-SC problems is to apply \cite{agarwal2016finding,carmon2018accelerated,carmon2017convex} to minimizing $\Phi$ while computing $\nabla \Phi$ via the Danskin's theorem. By the $O(\kappa_y\ell)$-smoothness of $\Phi$, it requires $\Tilde{O}\big(\sqrt{\kappa_y}\ell^{1/2}L_2^{1/4}\epsilon^{-7/4}\big)$ evaluations of $\nabla \Phi$ to find an $\epsilon$-stationary point, while each approximate evaluation of $\nabla \Phi(x)$ requires $\Tilde{O}\big(\sqrt{\kappa_y}\ln\epsilon^{-1}\big)$ computation of $\nabla_y f(x,\cdot)$. Overall, an $\tilde{O}\big(\kappa_y\ell^{1/2}L_2^{1/4}\epsilon^{-7/4}\big)$ complexity is achieved. On the other hand, when studying the lower bound of NC-SC problems with second-order smoothness, the hard instance we construct indicates that $\Omega\big(\sqrt{\kappa_y}\ell^{3/7}L_2^{2/7}\epsilon^{-12/7}\big)$ gradient evaluations are required for first-order methods, for sufficiently small $\epsilon$ (Section \ref{lower}). Compared to the best-known $\tilde{O}\big(\sqrt{\kappa_y}\ell\epsilon^{-2}\big)$ complexity \cite{lin2020near,zhang2021complexity}, our lower bound indicates a clear $\tilde{O}(\epsilon^{-2/7})$ sub-optimality gap in the existing results as $\epsilon\to0$. For finite $\epsilon$, the upper bound \cite{lin2020near,zhang2021complexity} differs from our lower bound by a factor of $O((\ell^2/\epsilon L_2)^{2/7})$. In the accuracy regime $\epsilon<O\big(\frac{\ell^2}{L_2}\big)$ where the lower bound becomes effective, in other words, the accuracy regime where the second-order smoothness starts to dominate the complexity, this sub-optimality gap is strictly greater than 1. On the contrary, the upper bound by naively applying \cite{agarwal2016finding,carmon2018accelerated,carmon2017convex} differs from \cite{lin2020near,zhang2021complexity} by a factor of $O((\ell^2/\epsilon L_2\kappa_y^2)^{1/4})$, which is a strict improvement only when $\epsilon<O(\frac{\ell^2}{L_2\kappa_y^2})$. Therefore, though the second-order smoothness dominates the optimal complexity in the accuracy regime $\Omega(\frac{\ell^2}{L_2\kappa_y^2})<\epsilon<O(\frac{\ell^2}{L_2})$, naively applying the existing results cannot exploit this benefit.

To resolve this issue, we first observe a weakly convex structure of $\phi$ introduced by the inner maximization of \eqref{prob:minimax}. Then, to exploit the weak-convexity, we carefully craft a monitored inexact proximal point algorithm framework based on the ``(assuming) convex until proven guilty'' (CUPG) \cite{carmon2017convex} technique. With an accelerated subproblem solver, our method achieves an $\tilde{O}\big(\sqrt{\kappa_y}\ell^{1/2}L_2^{1/4}\epsilon^{-7/4}\big)$ complexity, which improves the state-of-the-art result \cite{lin2020near,zhang2021complexity} by an $\tilde{O}(\epsilon^{-1/4})$ factor as $\epsilon\to0$. For finite $\epsilon$ in the effective accuracy regime where second-order smoothness dominates, our result is suboptimal by a small factor of $\tilde{O}\big(\big(\frac{\ell^2}{L_2\epsilon}\big)^{1/28}\big)$ compared to the our lower bound, while achieving an $\tilde{O}\big(\big(\frac{\ell^2}{L_2\epsilon}\big)^{1/4}\big)$ improvement compared to \cite{lin2020near,zhang2021complexity}. \vspace{0.2cm}

\newpage
\noindent\textbf{Contributions.} We summarize the contribution of this paper as follows. \vspace{0.1cm}
\begin{itemize}[leftmargin=0.8cm]
    \item Restricted by the gradient descent structure, the original CUPG method \cite{carmon2017convex} is unable to exploit the weak-convexity of $\Phi$ introduced by the minimax structure of \eqref{prob:minimax}. To resolve this issue, we design an \underline{I}nexact \underline{AP}PA \underline{U}ntil \underline{N}onconvexity (IAPUN) algorithm, which is an inexact accelerated proximal point variant of CUPG that allows inexact first-order evaluations. The inexact nature of our method further relaxes the need for exact function and gradient evaluations, which is crucial in the original CUPG for the progress checking purpose.
    \item We apply the IAPUN method to the deterministic NC-SC minimax problem \eqref{prob:minimax} and achieve an $\tilde{O}\big(\sqrt{\kappa_y}\ell^{1/2}L_2^{1/4}\epsilon^{-7/4}\big)$ complexity by exploiting the weak-convexity of $\Phi$. Due to our ability to handle inexactness, we derive a variance reduced version of our method to obtain an $\tilde{O}\big(\big(m+m^{\frac{3}{4}} \sqrt{\kappa_y}\big)\ell^{1/2}L_2^{1/4} \epsilon^{-7/4}\big)$ sample complexity for the finite-sum setting with $m$ component functions, improving its deterministic counterpart by a factor of $O(m^{-1/4})$.
    \item For the NC-SC minimax problems with second-order smoothness, we derive an $\Omega\big(\sqrt{\kappa_y}\ell^{3/7}L_2^{2/7}\epsilon^{-12/7}\big)$ lower bound for first-order methods when $\epsilon$ is small enough, proving an almost negligible $\tilde{O}\big(\big(\frac{\ell^2}{L_2\epsilon}\big)^{1/28}\big)$ sub-optimality gap of our algorithm.  It is also worth noting our hard instance can be easily extended to provide lower bounds for NC-SC problems with general $p$-th order smoothness.  
\vspace{0.1cm}
\end{itemize} 
Besides, we also apply the IAPUN approach to linearly constrained minimax problems and bilevel optimization problems and derive new convergence and complexity results. In particular, our $\tilde{O}(\epsilon^{-7/4})$ complexity for bilevel optimization was %also rediscovered by 
also derived in a parallel work \cite{yang2023accelerating} which was released 2 months after ours.\vspace{0.2cm}\\
\noindent\textbf{Related Works.} This paper is closely related to several important topics in nonconvex optimization, minimax optimization, and bilevel optimization problems, where the related works are listed below according to their topics. 

\paragraph{Nonconvex optimization} For smooth nonconvex optimization with gradient Lipschitz continuity, it is long known that the standard gradient descent method achieves the optimal $O(\epsilon^{-2})$ complexity \cite{bertsekas1997nonlinear,carmon2020lower} for first-order method. Given access to higher-order oracles, under appropriate continuity conditions, \cite{nesterov2006cubic} derived a cubic regularized Newton method with $O(\epsilon^{-3/2})$ oracle complexity, and is further extended to a general $p$-th order method an $O\big(\epsilon^{-(p+1)/p}\big)$ oracle complexity \cite{birgin2017worst}. Lower bounds for $p$-th order methods are shown in \cite{carmon2020lower}, proving the optimality of \cite{nesterov2006cubic,birgin2017worst}. Another line of research, including our work, has been focusing on exploiting higher-order smoothness with pure first-order methods. Several works have achieved the $\tilde{O}\big(\epsilon^{-7/4}\big)$ complexity under second-order smoothness \cite{agarwal2016finding,carmon2018accelerated,carmon2017convex}, which is sub-optimal by an almost negligible $\tilde{O}\big(\epsilon^{-1/28}\big)$ factor compared to the lower bound result \cite{carmon2021lower}.

\paragraph{Nonconvex minimax problem} 
There are two important settings under this topic, the NC-SC setting considered in this paper and the nonconvex-concave (NC-C) setting. Many papers have studied this setting \cite{lin2020gradient,lin2020near,lu2020hybrid,thekumparampil2019efficient,kong2021accelerated,boct2023alternating}, yet the lack of strong-concavity prevents the smoothness or even the existence of $\nabla\Phi$ and $\nabla^2\Phi$, which is out of the scope of this paper. For the NC-SC setting, several convergent first-order algorithms have been designed and analyzed \cite{lin2020gradient,lin2020near,lu2020hybrid,zhang2021complexity,boct2023alternating}. In particular, given first-order smoothness conditions, \cite{lin2020near,zhang2021complexity} have achieved an optimal complexity of $\tilde{O}\big(\sqrt{\kappa_y}\ell\epsilon^{-2}\big)$, matching the lower bound provided by \cite{zhang2021complexity}. The finite-sum NC-SC setting has also been studied in \cite{zhang2021complexity,luo2020stochastic}, where an $\tilde{O}\big(\!\big(m+\!\sqrt{\kappa_y}m^{\frac{3}{4}} \big) \epsilon^{-2}\big)$ sample complexity upper bound and an $\Omega\left(\!\big(m+\sqrt{\kappa_ym}\big) \epsilon^{-2}\right)$ lower bound are established in \cite{zhang2021complexity}. Though there is a rich bulk of literature on higher-order methods for saddle point problems \cite{luo2021finding,bullins2022higher,lin2023monotone,lin2022perseus,jiang2022generalized,chen2022cubic,huang2022cubic} as listed in the introduction, to our best knowledge, there is no existing first-order method that can harness the higher-order smoothness of the minimax problems without computing higher-order derivatives.

\paragraph{Bilevel optimization} This topic has various different setups, among which the most related one is the NC-SC setting. Under this setting, an $\tilde{O}(\epsilon^{-5/2})$ complexity was first derived in \cite{ghadimi2018approximation}, and was later improved to $\tilde{O}(\epsilon^{-2})$ in \cite{ji2021bilevel,chen2023near}. Given the third-order smoothness assumptions, \cite{yang2023accelerating} improved the complexity to $\tilde{O}(\epsilon^{-7/4})$ by applying an restarted acceleration method \cite{li2023restarted}. Besides, there are also several results on escaping saddle points of NC-SC bilevel optimization \cite{huang2022efficiently,chen2023near} with perturbed gradient method \cite{jin2017escape}, but we do not expand this discussion as it is beyond the scope of this paper.\vspace{0.2cm} 

\noindent\textbf{Outline.}  
In Section \ref{alg}, we describe the IAPUN framework for general weakly convex optimization problems and analyze its complexity for obtaining $\epsilon$-stationary points. And then, we discuss the consequence of IAPUN in both deterministic and finite-sum NC-SC minimax problems and obtain the improved complexities. In Section \ref{lower}, a hard NC-SC instance with second-order smoothness is constructed and an almost tight complexity lower bound is derived for the first-order methods. Finally, in Section \ref{application}, we discuss some further applications of IAPUN on linearly constrained minimax problems and bilevel optimization problems. We conclude the paper in Section \ref{conclusion}.

\section{IAPUN and its application to NC-SC problems} \label{alg}
\subsection{Fundamental assumptions}\label{subsec:prelim}
In this section, we formally state the basic assumptions of this paper, where we mainly focus on the NC-SC problems.  

\begin{assumption}\label{assu:minimax1}
The function $f$ is nonconvex and $\mu$-strongly concave. That is, $f(x,\cdot)$ is $\mu$-strongly concave in $y$ for $\forall x$ and $f(\cdot, y)$ may be nonconvex in $x$ for $\forall y$.
\end{assumption}
\begin{assumption}\label{assu:minimax2}
The function $\Phi(x):=\max_yf(x,y)$ is lower bounded. And the function $f$ is $\ell$-smooth, namely, $\nabla f$ is $\ell$-Lipschitz continuous in $(x,y)$.
\end{assumption}
As the NC-SC problem \eqref{prob:minimax} is equivalent to $\min_{x} \Phi(x)$, we introduce the following Lemma characterizes the first-order smoothness of $\Phi(x)$. 

\begin{lemma}\label{lemma:minimax}
Let $L_1\!=\!(1+\kappa_y)\ell$ and $\kappa_y\!=\!\ell/\mu$. Then, $\Phi$ is $L_1$-smooth and $\ell$-weakly convex under Assumptions \ref{assu:minimax1} and \ref{assu:minimax2}, where we say a function $g$ is $\gamma$-weakly convex ($\gamma>0$) if the sum of $g$ and a $\gamma$-strongly convex function is convex. 
\end{lemma}
\begin{proof} First, the $L_1$-smoothness of $\Phi$ follows directly from \cite[Lemma 23]{lin2020near}. Second, the $\ell$-smoothness of $f$ indicates the convexity of $f(\cdot,y)+\frac{\ell}{2}\|\cdot\|^2$  for any $y$. As the maximum of a family of convex functions, $\Phi(x)+\frac{\ell}{2}\|x\|^2$ is convex.  \end{proof}
Finally, we make the following assumption on the second-order smoothness of $\Phi$.

\begin{assumption} \label{assu:minimax3} The function $\Phi(x)$ is twice continuously differentiable and the Hessian matrix $\nabla^2 \Phi(x)$ is $L_2$-Lipschitz continuous in $\mathbb{R}^{n}.$
\end{assumption}
This assumption is very natural. Suppose $\nabla^2f$ is $\ell_2$-Lipschitz continuous, then \cite[Proposition 1]{chen2022cubic} provides a pessimistic bound that $L_2\leq (1+\kappa_y)^3\ell_2$. However, $L_2$ can be much smaller in many situations. For example, we even have $L_2=\ell_2$ in the hard instance for lower bound construction in Section \ref{lower}. As $\ell_2$ does not directly affect the algorithm design and analysis, we would make Assumption \ref{assu:minimax3} directly on $L_2$ instead of $\ell_2$. Nevertheless, one may still use the upper estimate $(1+\kappa_y)^3\ell_2$ in the algorithm if a reasonable estimate of $L_2$ is not available.

\subsection{The IAPUN algorithm}
\label{section:algo}
To motivate the algorithm, we first briefly discuss the insights behind the optimal method without exploiting  second-order smoothness \cite{lin2020near,zhang2021complexity}, which is based on the proximal point algorithm (PPA): \vspace{-0.09cm}
\begin{equation}
    \label{eqn:ppa} x_{k+1}=\mathop{\text{argmin}}_x \,\, \Phi(x)+\gamma\|x-x_k\|^2.\vspace{-0.13cm}
\end{equation}
It is well-known that PPA takes $O(\gamma\epsilon^{-2})$ iterations to find an $\epsilon$-stationary point, while the $\ell$-weak convexity of $\Phi(x)$ allows us to take an aggressive proximity parameter $\gamma=\ell$ while maintaining the solvability of subproblems. This is a clear contrast to gradient descent (GD) that needs $O(\kappa_y\ell\epsilon^{-2})$ iterations to converge, as GD requires a stepsize $\eta\leq 1/L_1$ for $L_1$-smooth functions. With $\gamma=\ell$ and  $\hat{f}_k(x,y) \!:=\! f(x,y)+\ell\|x-x_k\|^2$, the proximal point subproblem \eqref{eqn:ppa} is equivalent to a minimax subproblem\vspace{-0.09cm} 
\begin{equation}
\label{eqn:ppa-minimax}
x_{k+1}=\mathop{\text{argmin}}_x \max_y \hat{f}_k(x,y),\vspace{-0.13cm}
\end{equation} 
whose objective function is $3\ell$-smooth, $\ell$-strongly convex, and $\mu$-strongly concave. Exchanging the min and max of \eqref{eqn:ppa-minimax} leads to maximizing $\Psi_k(y):=\min_x\hat{f}_k(x,y)$. As the condition number of $\hat{f}(\cdot,y)$ is $O(1)$ in $x$,  we know $\Psi_k(y)$ is still $O(\ell)$-smooth and $\mu$-strongly concave. Note that $\nabla\Psi_k(y)$ can be evaluated by solving $\min_x \hat{f}_k(x,y)$ to sufficient accuracy and applying Danskin's theorem. If both $\max_y\Psi_k(y)$ and $\min_{x} \hat f_k(x,y)$ are solved by Nesterov's accelerated gradient method \cite{nesterov2018lectures}, it will take $\Tilde{O}\big(\sqrt{\kappa_y}\ln\frac{1}{\epsilon}\big)$ gradient evaluations to inexactly compute each PPA step, leading to a first-oracle complexity of $\Tilde{O}\big(\sqrt{\kappa_y}\ell\epsilon^{-2}\big)$ for NC-SC problem. 

Next, to further exploit the second-order smoothness to accelerate the algorithm, we propose to adopt a proximal point variant of the CUPG technique \cite{carmon2017convex}. We start by presuming $\Phi$ to be convex and then apply the accelerated proximal point algorithm to a sequence of carefully crafted subproblems with $\gamma\ll\ell$. If $\Phi$ behaves convex along the iterations, then the fast convergence of convex optimization is achieved. If $\Phi$ behaves nonconvex at some iteration, then a certificate of negative curvature will be returned and a sufficiently large descent can be achieved. %Overall, an improved complexity can be achieved if the objective function possesses higher-order smoothness properties. 
Below, we formally describe the IAPUN algorithm for solving $\min_x \Phi(x)$ and the required assumptions. 
\begin{assumption}\label{assu:Phi}
Function $\Phi(\cdot)$ is lower bounded, $\ell$-weakly convex, $L_1$-smooth, and it has globally $L_2$-Lipschitz continuous Hessian matrix $\nabla^2 \Phi(\cdot)$. 
\end{assumption}
\noindent If $\Phi(\cdot)$ originates from the NC-SC minimax problem \eqref{prob:minimax}, then Assumption \ref{assu:Phi} is implied by Assumptions \ref{assu:minimax1}, \ref{assu:minimax2} and \ref{assu:minimax3}. In this case,  $\ell$ is often much smaller than $L_1$. Note that in many settings the approximate evaluations of $\Phi(x)$ and $\nabla \Phi(x)$ are efficiently accessible while the exact values cannot be obtained. For generality, we make the following assumption on the inexact estimation of $\Phi$ and $\nabla\Phi$.  
\begin{assumption}\label{error}
For $\forall \Delta_y,\delta_y\!>\!0$ and $\forall x$, we can obtain approximate function and gradient estimators $\phi_x\!$ and $g_x$ s.t. $\left|\phi_{x}-\Phi(x)\right|\le \delta_y$ and $\left\|g_x-\nabla\Phi(x)\right\|\le \Delta_y.$
\end{assumption}

We describe the IAPUN method as Algorithm \ref{alg:main}. Similar to \cite{lin2020near,zhang2021complexity}, at each outer iteration $p_{k-1}$, we attempt to solve the proximal point subproblem (Line 5 - Line 16): \vspace{-0.4cm}
\begin{equation}
\label{prob:PPA-IAPUN}
\min_x\,\,\,\hat{\Phi}_k(x):=\Phi(x)+\alpha \|x-p_{k-1}\|^2,\vspace{-0.13cm}
\end{equation} 
where $\alpha\ll\ell$. Despite the nonconvexity of $\hat{\Phi}_k(x)$, IAPUN executes the inexact accelerated proximal point algorithm (APPA) to minimize it (Line 8 - Line 9). If $\hat{\Phi}_k(x)$ behaves convex along the iterations, then the progress checking (Line 10) in the $\texttt{Certify}$ function (Algorithm \ref{alg:Certify}) will observe a fast convergence of convex optimization and we successfully perform an inexact PPA iteration \eqref{prob:PPA-IAPUN} with sufficient descent (Line 12 with Flag=5), otherwise IAPUN either observes a sufficient descent (Line 12 with Flag = 3 \& Line 15) or can perform a negative curvature descent without computing Hessian (Line 16) by the $\texttt{Exploit-Ncvx}$ function (Algorithm \ref{alg:Exploit-Ncvx}). With function and gradient estimators satisfying Assumption \ref{error}, we adopt the notations\vspace{-0.05cm}   
$$\hat{\phi}_{k(x)}:=\phi_{x} + \alpha\|x-p_{k-1}\|^2\qquad\mbox{and}\qquad\hat{g}_{k(x)}:=g_x + 2\alpha(x-p_{k-1})\vspace{-0.05cm}$$ 
in Algorithm \ref{alg:Exploit-Ncvx} and Algorithm \ref{alg:Certify} for the ease of presentation.

\begin{algorithm2e}
\caption{IAPUN} % Inexact APPA until nonconvexity 
\label{alg:main}
\textbf{Input:} Initial point $p_0$, strong concavity modulus $\mu>0$, Lipschitz constants $\ell,L_2>0$, proximal coefficients $\alpha>0,\gamma>\ell$,  and tolerances $\delta_x,\delta_y,\epsilon>0$ \\
Set $\eta=\frac{\alpha}{L_2}$, $\kappa_x \!=\! \frac{\gamma}{\alpha}$, $\omega \!=\! \frac{2\sqrt{\kappa_x}-1}{2\sqrt{\kappa_x}+1}$, $\kappa_y=\frac{\ell}{\mu},d = \frac{\alpha}{L_2}$, and 
$~\qquad\qquad\qquad~$$\chi \!=\! 6\sqrt{\kappa_x}\cdot\left(11\kappa_x\delta_x+(2\ell+L_1+\alpha)\sqrt{\frac{2\delta_x}{\gamma+2\alpha}}\cdot d\right)$.\vspace{0.1cm}\\
\For{$k=1,2,3,\cdots$}{
    Define $\hat{\Phi}_k(x)=\Phi(x)+\alpha \|x-p_{k-1}\|^2$ and set $\tilde{x}_{k,0}=x_{k,0}=p_{k-1}$.\\
    \For{$t = 1,2,3,\cdots$}{/*** Line 8-9 is the inexact proximal point method. ***/\\ 
    Update the epoch length as $T_k = t$.\\
    Find $x_{k,t}\approx\mathop{\mathrm{argmin}}_{x} \hat{\Phi}_k(x)+\gamma \|x-\tilde{x}_{k,t-1}\|^2$, such that $\hat{\Phi}_k(x_{k,t})+\gamma \|x_{k,t}-\tilde{x}_{k,t-1}\|^2\le\min_x \hat{\Phi}_k(x)+\gamma \|x-\tilde{x}_{k,t-1}\|^2+\delta_x$.\\
    Set $\txk = x_{k,t} + \omega(x_{k,t} - x_{k,t-1})$. \\
    Compute $(\text{Flag},w^x_k) = \texttt{Certify}(\hpk,t,x_{k,t},x_{k,0}).$\\
    \textbf{if} \,$\text{Flag}\neq\textbf{null}\,$ \textbf{then}\,  \textbf{break} the forloop. 
}
\textbf{if} $\text{Flag} == 3 \text{ or } 5$, \,\,\textbf{then}\,\,  Set $p_k = w^x_k$. \\
\If{$\emph{Flag}==1, 2 \emph{ or } 4$,}{
    Evaluate $\phi_{z^x}\!:\! z^x\!\in\!\{w^x_k\}\!\cup\!\{x_{k,t}\!\}_{t=0}^{T_k}$ satisfying Assumption \ref{error}. Set \vspace{-0.15cm}
    $$z^x_k = \mathop{\mathrm{argmin}}_{z^x} \big\{\!\phi_{z^x}: z^x\!\in\!\{w^x_k\}\!\cup\!\{x_{k,t}\!\}_{t=0}^{T_k}\big\}.\vspace{-0.25cm}$$\\
    \textbf{if} $\phi_{z^x_k}<\phi_{x_{k,0}}-\frac{\alpha^3}{32L_2^2}\,$ \textbf{then} Set $p_k = z^x_k$.\\
    \textbf{else} $p_k = \texttt{Exploit-Ncvx}\big(\Phi,\{w^x_k\}\cup\{x_{k,t}\}_{t=0}^{T_k}\big).$
} 
Approximately evaluate $g_{p_k}$  satisfying Assumption \ref{error}.\\
\textbf{if}\,\,$\|g_{p_k}\|\leq \frac{3\epsilon}{4}$ \,\,\textbf{then} \,\,Return($p_k$).
} \vspace{-0.1cm}
\end{algorithm2e} \vspace{-0.1cm}
Algorithm \ref{alg:Certify} can return many flag values to represent different scenarios that can happen during the algorithm. In detail, the first three cases happen when we found the approximate majorized function value increases a lot. $\text{Flag}=1$ means a sufficient descent is obtained. $\text{Flag}=2$ shows the iterate is close to the starting point while the function increases. In this case, an NC pair exists and it serves as a certificate of detecting negative curvature of $\Phi$. If $\text{Flag}=3$, a descent similar to the case $\text{Flag}=1$ can be proved. If the approximate majorized function value does not increase a lot, we check the iterate by inexactly solving $\min_{x} \hat{\Phi}_k(x)+\gamma \|x-x_{k,t}\|^2$. If $\text{Flag}=4$, the approximate solution $w_k^x$ is far from $x_{k,t}$. However, if $\hat{\Phi}$ behaves convex along the trajectory, then $\gamma\|w_k^x-x_{k,t}\|^2$ should be well controlled by the compared value. Thus $\hat{\Phi}$ must behave nonconvex along the trajectory and we can detect an NC-pair. For $\text{Flag}=5$, the sufficient descent analysis is similar to the case $\text{Flag}=3$. 

\begin{algorithm2e}
\caption{$p_k = \texttt{Exploit-Ncvx}(\hpk,\Phi,\{w^x_k\}\cup\{x_{k,t}\}_{t=0}^{T_k})$}
\label{alg:Exploit-Ncvx} 
\textbf{Additional input:} Parameters $\delta_y,\kappa_y,\ell,\alpha,\eta$ are inherited from Algorithm \ref{alg:main}.\\
Denote $\zeta(x,x' ):= \hat{\phi}_{k(x)}-\hat{\phi}_{k(x')}-\hat{g}_{k(x')}^{\top}(x-x')-\frac{\alpha}{2}\|x-x'\|^2$.\\
\For{$t=1,2,\cdots T_k$}{
    Compute $\zeta(x_{k,t-1},x_{k,t})$ after approximately evaluating $\phi_{x_{k,t-1}}$, $\phi_{x_{k,t}}$, and $g_{x_{k,t}}$ satisfying Assumption \ref{error}.\\
    \If{$\zeta(x_{k,t-1},x_{k,t})<-2\delta_y - \Delta_y\|x_{k,t-1}-x_{k,t}\|$}{Set $(u,v) = (x_{k,t-1},x_{k,t})$ and \textbf{break} the forloop.}
    Compute $\zeta(w_k^x,x_{k,t})$ after approximately evaluating $\phi_{w_k^x}$, $\phi_{x_{k,t}}$, and $g_{x_{k,t}}$ satisfying Assumption \ref{error}.\\
    \If{$\zeta(w_k^x,x_{k,t})<-2\delta_y - \Delta_y\|w_k^x-x_{k,t}\|$}{Set $(u,v) = (w_k^x,x_{k,t})$ and \textbf{break} the forloop. }
    \textcolor{gray}{\,\,\,/*** We will theoretically prove that a pair $(u,v)$ can be found. ***/}
}
Set $x^1 = u + \eta\cdot\frac{(u-v)}{\|u-v\|}$ and $x^2 = u - \eta\cdot\frac{(u-v)}{\|u-v\|}$.\\
Return $p_k = \mathop{\mathrm{argmin}}_x\left\{\phi_{x}:x\in\{x^1,x^2\}\right\}$. 
\end{algorithm2e} 

\subsection{Complexity of IAPUN}
In this section, we present the complexity analysis of IAPUN. Starting by assuming $\hat{\Phi}_k(x)$ to be $\alpha$-strongly convex, a linear convergence is expected. If it turns out not to be the case,  we can then find a negative curvature, and thus apply Algorithm \ref{alg:Exploit-Ncvx} to decrease the objective value. As a result, we can guarantee sufficient improvement in each epoch. The lemma below guarantees the inner loop of Algorithm \ref{alg:main} (Line 4 - Line 18) to stop in $\tilde{O}\big(1+\sqrt{\frac{\gamma}{\alpha}}\big)$ steps.

\begin{lemma}
\label{lemma:epln} Consider the $k$-th epoch $\{x_{k,t}\}_{t=0}^{T_k}$ and the point $w^x_k$ generated by Algorithm \ref{alg:main}. Suppose the parameters are chosen so that $\chi+\delta_x+\delta_y\leq\frac{\epsilon^2}{3200\gamma}$, then 
\begin{equation}
    \label{prop:epln-01}
    T_k\leq 1 + 6\sqrt{\frac{\gamma}{\alpha}}\cdot\max\left\{0,\log\left(\frac{3200\gamma(\Phi(x_{k,0})-\Phi^*+2\delta_y)}{\epsilon^2}\right)\right\}.
\end{equation} 
\end{lemma}

\begin{proof}
Suppose the epoch $k$ terminates at the $T_k$-th iteration. Then according to Line 11 of Algorithm \ref{alg:main}, Algorithm \ref{alg:Certify} must return $\text{Flag}=\textbf{null}$ for the first $T_k-1$ iterations. According to Line 15, Line 17, and Line 19 in Algorithm \ref{alg:Certify}, for  $1\leq t\leq T_k-1$, the following inequalities must hold:
\begin{equation}
    \label{prop:epln-1}
    \gamma\|w^x_k-x_{k,t}\|^2\leq\big(1-\frac{1}{6\sqrt{\kappa_x}}\big)^tE_k + \chi +\delta_x+\delta_y \qquad \mbox{and}\qquad\gamma\|w^x_k-x_{k,t}\|\geq \frac{\epsilon}{40},
\end{equation}  
which further indicates that
\begin{equation} 
    \label{prop:epln-2}
    e^{-\frac{t}{6\sqrt{\kappa_x}}}E_k\geq\big(1-\frac{1}{6\sqrt{\kappa_x}}\big)^tE_k \geq \frac{\epsilon^2}{1600\gamma}-\chi - \delta_x-\delta_y\geq \frac{\epsilon^2}{3200\gamma}.
\end{equation}  
Under Assumption \ref{error}, an upper bound for $E_k$ at any $t\leq T_k-1$ satisfies
\begin{eqnarray}
    \label{prop:epln-3}
    E_k & = & \hat{\phi}_{k(x_{k,0})} - \hat{\phi}_{k(w^x_k)} + \frac{\alpha}{4}\|w^x_k-x_{k,0}\|^2\\
    & \leq & \Phi(x_{k,0})+\delta_y - (\Phi(w^x_k) + \alpha\|w^x_k-x_{k,0}\|^2-\delta_y) + \frac{\alpha}{4}\|w^x_k-x_{k,0}\|^2\nonumber\\
    & \leq & \Phi(x_{k,0}) - \Phi^* + 2\delta_y.\nonumber
\end{eqnarray} 
Combining the inequalities \eqref{prop:epln-2} and \eqref{prop:epln-3} and taking $t=T_k-1$ proves \eqref{prop:epln-01}. 
\end{proof}

\begin{algorithm2e}
    \caption{$(\text{Flag},w^x_k,w^y_k) = \texttt{Certify}(\hpk,t,x_{k,t},x_{k,0})$}
	\label{alg:Certify} 
	\textbf{Additional input:}  $\delta_x,\delta_y,\chi,\alpha,\gamma,\kappa_x,\epsilon$ are inherited from Algorithm \ref{alg:main}.\\
    Approximately evaluate $\phi_{x_{k,t}}$, and $\phi_{x_{k,0}}$ satisfying Assumption \ref{error}.\\
\If{$\hat{\phi}_{k(x_{k,t})}>\hat{\phi}_{k(x_{k,0})}+\chi+2\delta_y$}{
	\If{\,$\min_{1\leq s\leq t-1} \,\phi_{x_{k,s}}\leq \phi_{x_{k,0}}-\frac{\alpha^3}{32L_2^2}$}{Return$(\text{Flag}=1, w^x_k = \textbf{null} )$}
    \uIf{$\|x_{k,0}-x_{k,t}\|\leq \frac{\alpha}{4L_2}$\vspace{0.1cm}}{
    Return$(\text{Flag}=2, w^x_k = x_{k,0})$}
    \Else{Recompute  $x_{k,t}\approx \mathop{\mathrm{argmin}}_{x \in \mathcal{X}_k} \hat{\Phi}_k(x)+\gamma \|x-\tilde{x}_{k,t-1}\|^2$ such that $x_{k,t}\in\mathcal{X}_k:=B(x_{k,0},\frac{\alpha}{4L_2})$, and \vspace{-0.3cm}
    $$\hat{\Phi}_k(x_{k,t})+\gamma \|x_{k,t}-\tilde{x}_{k,t-1}\|^2\le\min_{x \in \mathcal{X}_k} \hat{\Phi}_k(x)+\gamma \|x-\tilde{x}_{k,t-1}\|^2+\delta_x$$\vspace{-0.45cm}}
     Approximately evaluate $\phi_{x_{k,t}}$ satisfying Assumption \ref{error}.\\
     \uIf{$\hat{\phi}_{k(x_{k,t})}>\hat{\phi}_{k(x_{k,0})}+\chi+2\delta_y $\vspace{0.1cm}}{Return$(\text{Flag}=2, w^x_k = x_{k,0})$ }
    \Else{Return$(\text{Flag} = 3$, $w^x_k = x_{k,t})$.\vspace{-0.1cm}}
	}
	Find $w^x_k\approx\mathop{\mathrm{argmin}}_{x} \hat{\Phi}_k(x)+\gamma \|x-x_{k,t}\|^2$, such that\vspace{-0.2cm} $$\hat{\Phi}_k(w_k^x)+\gamma \|w_k^x-x_{k,t}\|^2\le\min_{x} \hat{\Phi}_k(x)+\gamma \|x-x_{k,t}\|^2+\delta_x$$\vspace{-0.45cm} \\
	 Approximately evaluate $\phi_{w^x_k}$ satisfying Assumption \ref{error}. Define $E_k:= \hat{\phi}_{k(x_{k,0})} - \hat{\phi}_{k(w^x_k)} + \frac{\alpha}{4}\|w^x_k-x_{k,0}\|^2$.\\
	\uIf{$\gamma\|w^x_k-x_{k,t}\|^2>\big(1-\frac{1}{6\sqrt{\kappa_x}}\big)^tE_k + \chi +\delta_x +2\delta_y$\vspace{0.1cm}}{Return($\text{Flag} = 4$, $w^x_k$).\vspace{0.1cm}}
	\ElseIf{$\gamma\|w^x_k-x_{k,t}\|\leq\frac{\epsilon}{40}$\vspace{0.1cm}}{Return($\text{Flag} = 5$, $w^x_k$).}  
	Return($\text{Flag} = \textbf{null}$, $w^x_k=\textbf{null}$). 
\end{algorithm2e}

The following two lemmas give the descent guarantee when $\text{Flag}=3 \text{ or } 5$, i.e., the descent when the algorithm does not exploit the negative curvature. The proof of which can be found in Appendices \ref{proof:lemma:Flag-5} and \ref{proof:lemma:Flag-3}.

\begin{lemma}\label{lemma:Flag-5}
Suppose Algorithm \ref{alg:main} does not terminate in the $k$-th epoch. Namely, we have $\|g_{p_k}\|>\frac{3\epsilon}{4}$. Suppose the algorithmic parameters are chosen so that $\Delta_y<\frac{\epsilon}{4}$, $\sqrt{\frac{2 \delta_x}{\gamma+2 \alpha}} \leq \frac{\epsilon}{20\left(L_1+2 \gamma\right)}$ and $\chi+\delta_x+4 \delta_y \leq \frac{\epsilon^2}{50 \alpha}$. If $p_k$ is given by Line 12 with $\text{Flag} = 5$ and $p_k=w^x_k$, then
\(\Phi\left(p_k\right)-\Phi\left(p_{k-1}\right) \leq-\frac{\epsilon^2}{50 \alpha}.\)
\end{lemma}

\begin{lemma}\label{lemma:Flag-3}
Suppose the algorithmic parameters are chosen to satisfy $\sqrt{\frac{2 \delta_x}{\gamma+2 \alpha}} \leq \frac{\alpha}{24 L_2}$ and $\chi+4 \delta_y \leq \frac{\alpha^3}{72 L_2^2}$. In the $k$-th epoch of Algorithm \ref{alg:main}, if $p_k$ is given by Line 12 with $\text{Flag} = 3$ and $p_k=w^x_k$, then
\(\Phi\left(p_k\right)-\Phi\left(p_{k-1}\right) \leq-\frac{\alpha^3}{72 L_2^2}.\)
\end{lemma}

\noindent Next, we show that the approximate majorized function $\hat{\phi}_{k(\cdot)}$ of the iteration points is upper bounded, which is guaranteed by the certifying process in Algorithm \ref{alg:Certify}.

\begin{lemma}\label{lemma:LeqFx0}
Consider the $k$-th epoch $\{x_{k,t}\}_{t=0}^{T_k}$ and the point $w^x_k$ generated by Algorithm \ref{alg:main}. For any $x\in \{x_{k,t}\}_{t=0}^{T_k-1}\cup \{w_k^x\}$, we have \vspace{-0.1cm}
\begin{equation}
    \label{prop:epln-02}
    \hat{\phi}_{k(x)}\leq \hat{\phi}_{k(x_{k,0})}+\chi+\delta_x +4\delta_y.\vspace{-0.1cm}
\end{equation}
\end{lemma}
We should emphasize that $x_{k,T_k}$ is not included in \eqref{prop:epln-02}.

\begin{proof}
First of all, let us notice that $\hpk(x_{k,0}) = \Phi(x_{k,0}) + \alpha\|x_{k,0}-x_{k,0}\|^2 = \Phi(x_{k,0})$. To show \eqref{prop:epln-02}, we should notice that when $x\in \{x_{k,t}\}_{t=0}^{T_k-1}$, \eqref{prop:epln-02} is directly guaranteed. This is because epoch $k$ does not terminate at $t\leq T_k-1$, meaning that Algorithm \ref{alg:Certify} returns Flag$=$\textbf{null}, and the condition in Line 3 of Algorithm \ref{alg:Certify} is not satisfied. When $x=w^x_k$, there are several situations. In case $\text{Flag}=1$ and $w^x_k= \textbf{null}$, thus \eqref{prop:epln-02} is not required for $w^x_k$ at all. In case $\text{Flag}=2$ and $w^x_k = x_{k,0}$, thus \eqref{prop:epln-02} holds trivially. In case $\text{Flag}=3$ and $w^x_k = x_{k,t}$, \eqref{prop:epln-02} also holds trivially by the condition for returning the flag value. In case $\text{Flag}=4$ or $5$, we know $\hat{\phi}_{k(x_{k,t})}\leq \hat{\phi}_{k(x_{k,0})} + \chi + 2\delta_y$, combined with Line 15 of Algorithm \ref{alg:Certify} and Line 8 of Algorithm \ref{alg:main}, we have \vspace{-0.1cm}
\begin{eqnarray*}
    \hat{\phi}_{k(w^x_k)} &\leq& \hat{\phi}_{k(w^x_k)} + \gamma\|w^x_k-x_{k,t}\|^2+\delta_y\leq \hpk(x_{k,t})+\delta_x+\delta_y\\
    &\leq&\hat{\phi}_{k(x_{k,t})} + \delta_x + 2\delta_y\leq \hat{\phi}_{k(x_{k,0})} + \chi  + \delta_x+ 4\delta_y,\vspace{-0.1cm}
\end{eqnarray*}
which completes the proof.  
\end{proof}

To exploit the negative curvature, the iteration points must stay in some $\Theta(\alpha/L_2)$-radius ball. We prove this property by the following Lemma.

\begin{lemma}\label{lemma:distance}
In Algorithm \ref{alg:main}, suppose the algorithmic parameters are chosen so that $\chi + \delta_x + 4\delta_y\leq\frac{\alpha^3}{32L_2^2}$. Then in any epoch $k$, if Line 16 is executed, then we must have $\|x-x_{k,0}\|\leq\frac{\alpha}{4L_2}$ for any $x\in\left\{x_{k,s}\right\}_{s=0}^{T_k}\!\cup\!\left\{w^x_k\right\}.$ Consequently, \vspace{-0.1cm} 
\begin{equation*}
    \max\left\{\|x_{k,s-1}-x_{k,s}\|,\|x_{k,s}-w^x_{k}\|\right\}\leq \frac{\alpha}{2L_2}, \qquad\mbox{for}\qquad 1\leq s \leq T_k.\vspace{-0.1cm}
\end{equation*}
\end{lemma}
Note that unlike Lemma \ref{lemma:LeqFx0}, Lemma \ref{lemma:distance} also characterizes the last iterate $x_{k,T_k}$.

\begin{proof}
First, consider $x\in \{x_{k,t}\}_{t=0}^{T_k-1}\cup \{w_k^x\}$, excluding the last point $x_{k,T_k}$.  If Line 16 of Algorithm \ref{alg:main} is executed, then  we must have  $\phi_{x}\geq \phi_{x_{k,0}}-\frac{\alpha^3}{32L_2^2}$. Combined with Lemma \ref{lemma:LeqFx0}, we have \vspace{-0.12cm}
\begin{equation*}
    \hat{\phi}_{k(x_{k,0})}+\chi+\delta_x +4\delta_y \geq \hat{\phi}_{k(x)}= \phi_{x} + \alpha\|x-x_{k,0}\|^2  \geq \phi_{x_{k,0}}-\frac{\alpha^3}{32L_2^2} + \alpha\|x-x_{k,0}\|^2.\vspace{-0.12cm}
\end{equation*}
Thus, we have $\alpha\|x-x_{k,0}\|^2\leq \chi + \delta_x + 4\delta_y +\frac{\alpha^3}{32L_2^2}\leq \frac{\alpha^3}{16L_2^2}$, which further indicates that $\|x-x_{k,0}\|\leq\frac{\alpha}{4L_2}$ for any $x\in\left\{x_{k,s}\right\}_{s=0}^{T_k-1}\cup\left\{w^x_k\right\}.$ \vspace{0.15cm}
	 
Next, we discuss the last point $x_{k,T_k}$. Note that when Line 16 is executed, $\text{Flag}=1$ is not possible.  If Algorithm \ref{alg:Certify} returns $\text{Flag}=2$ or $3$,  then Line 6 to Line 9 of Algorithm \ref{alg:Certify} forcefully impose $\|x_{k,T_k}-x_{k,0}\|\leq\frac{\alpha}{4L_2}$. If Algorithm \ref{alg:Certify} returns $\text{Flag}=4$ or $5$, then we know $\hat{\phi}_{k(x_{k,T_k})}\le\hat{\phi}_{k(x_{k,0})}+\chi+2\delta_y$. Together with the necessary condition for executing Line 16 of Algorithm \ref{alg:main}, namely, $\phi_{x_{k,T_k}}\geq \phi_{x_{k,0}}-\frac{\alpha^3}{32L_2^2}$, we can also derive that $\|x_{k,T_k}-x_{k,0}\|\leq\frac{\alpha}{4L_2}$.  Combining the above results proves the lemma. 
\end{proof}
The next lemma describes the iterations of APPA, which involves a quantity $\tilde{D}_k$ that can help detect the negative curvature. The proof can be found in Appendix \ref{Proof:lemma:aapp-intermediate}.

\begin{lemma}
\label{lemma:aapp-intermediate}
Let us reuse the notation $\cX_k=B(x_{k,0},\frac{\alpha}{4L_2})$ from Algorithm \ref{alg:Certify}. For any $x,x'\in\cX_k$, let us define the quantity 
$$\tilde{D}_k(x,x')\!:=\!\hpk(x)\!-\!\hpk(x') \!-\! \nabla\hpk(x')^{\!\top}\!(x\!-\!x')\!-\!\frac{\alpha}{2}\|x\!-\!x'\|^2 \!+\! \delta_x \!+\! \sqrt{\frac{2\delta_x}{\gamma\!+\!2\alpha}}\cdot\ell\|x\!-\!x'\|.$$
Suppose we set $\gamma\geq\ell$. Then, for any $\txkm$, $x_{k,t}$ generated by Algorithm \ref{alg:main}, and any point $\bar x\in\cX_k$, the following inequality holds
\begin{align}
    &\hpk(\bar x) \geq  \hpk(x_{k,t}) \!-\! 2\gamma(\bar x \!-\!\txkm)^{\top}(x_{k,t} \!-\! \txkm) \!+\! 2\gamma\|x_{k,t}\!-\!\txkm\|^2\nonumber\\
    &\qquad + \frac{\alpha}{4}\|\bar x \!-\! x_{k,t}\|^2 \!+\! \tilde{D}_k(\bar x, x_{k,t}) \!-\! 11\kappa_x\delta_x \!-\! (2\ell+L_1+\alpha)\sqrt{\frac{2\delta_x}{\gamma\!+\!2\alpha}}\cdot\|\bar x \!-\! x_{k,t}\|.
\end{align} 
\end{lemma}

\noindent To understand the quantity $\tilde{D}_k(x,x')$, assume $\delta_x=0$. Then, the quantity $\tilde{D}_k(x,x')$ is non-negative if $\hat{\Phi}_k$ is $\alpha$-strongly convex. Having $\tilde{D}_k(x,x')<0$ indicates the existence of the negative curvature in $\nabla^2\hat\Phi$. When $\delta_x>0$, the last two terms are added to make $\tilde{D}_k(x,x')$ inclusive with the inexact solution of the subproblems. With the above lemma, if $\hat{\Phi}_k$ is strongly convex, we expect to have linear convergence when $\tilde{D}_k(x,x')\geq 0$ always hold for the iteration points. As a result, the following Lemma depicts one epoch of Algorithm \ref{alg:main}. See analysis in Appendix \ref{proof:lemma:aapp}. 

\begin{lemma}\label{lemma:aapp}
Consider the $k$-th epoch  $\{x_{k,t}\}_{t=0}^{T_k}$ generated by Algorithm \ref{alg:main}. For any point $\bar x\in\cX_k$, let $d>0$ be a constant s.t. $d\geq\max\left\{\|x_{k,s-1}-x_{k,s}\|,\|\bar x-x_{k,s}\|\right\}$ for $s = 1,2,\cdots,T_k$. For any $t\leq T_k$, if $\tilde{D}_k(x_{k,s-1},x_{k,s})\geq0$ and $\tilde{D}_k(\bar x, x_{k,s})\geq0$ hold for $1\leq s\leq t$, then 
$$\hpk(x_{k,t}) - \hpk(\bar x) \leq \left(1-\frac{1}{6\sqrt{\kappa_x}}\right)^t\cdot\left(\hpk(x_{k,0})-\hpk(\bar x) + \frac{\alpha}{4}\|\bar x-x_{k,0}\|^2\right) + \chi.$$
\end{lemma}

If the situation in the above lemma happens, that is $\tilde{D}_k(x,x')\geq 0$ always holds for the iteration points, a fast linear convergence happens at Flag=3,5, and enough descent is guaranteed. However, if it does not happen, then we observe a certificated nonconvex pair along the trajectory. Then, we will execute Line 16.

The following lemma discusses the situation when the progress checking returns the $\text{Flag}=$ 2 or 4. In these situations, it is guaranteed to find a nontrivial violation of convexity to perform negative curvature descent, see analysis in Appendix \ref{proof:lemma:Flag-24}.

\begin{lemma}\label{lemma:Flag-24}
Suppose we set the parameters s.t. $\chi + \delta_x + 4\delta_y\leq\frac{\alpha^3}{32L_2^2}$, $\delta_x\ge 2\delta_y$ and $\ell\sqrt{\frac{2 \delta_x}{\gamma+2 \alpha}} \geq 2 \Delta_y$.  In the $k$-th epoch of the Algorithm \ref{alg:main}, suppose Line 16 is executed. Then Algorithm \ref{alg:Exploit-Ncvx} will find some $u,v\in\cX_k$ s.t. $\Phi(u)\leq \Phi(x_{k,0}) + \cE+\delta_x+6\delta_y$ and 
\[\hat{\Phi}_k(u)-\hat{\Phi}_k(v)-\nabla \hat{\Phi}_k(v)^{\top}(u-v)-\frac{\alpha}{2}\|u-v\|^2<0,\]
which is equivalent to 
\[\Phi_k(u)-\Phi_k(v)-\nabla \Phi_k(v)^{\top}(u-v)+\frac{\alpha}{2}\|u-v\|^2<0.\]
\end{lemma}

\noindent We are guaranteed to find a nontrivial violation of convexity. The following lemma shows that a guarantee of the function value decreases by using negative curvature.

\begin{lemma}{\cite[Lemma 1]{carmon2017convex}}\label{lemma:NC}
Suppose $\nabla^2\Phi$ is $\!L_2$-Lipschitz, $\alpha\!>\!0$, and $u,v$ satisfy 
\[\Phi_k(u)-\Phi_k(v)-\nabla \Phi_k(v)^{\top}(u-v)+\frac{\alpha}{2}\|u-v\|^2<0.\]
If $\|u-v\|\leq \frac{\alpha}{2L_2}$, and $x^1=u + \eta\cdot\frac{(u-v)}{\|u-v\|}, x^2=u - \eta\cdot\frac{(u-v)}{\|u-v\|}$, then 
 we have
\[\min\{\Phi(x^1),\Phi(x^2)\}\leq \Phi(u)-\frac{\alpha\eta^2}{12} \quad\mbox{for}\quad \forall\eta\leq \frac{\alpha}{L_2}.\]
\end{lemma}

\noindent Combining Lemma \ref{lemma:Flag-24} and Lemma \ref{lemma:NC} by choosing $\eta=\frac{\alpha}{L_2}$, we obtain
\[\min\{\Phi(x^1),\Phi(x^2)\}\leq \Phi(u)-\frac{\alpha^3}{12L_2^2}\leq \Phi\left(x_{k, 0}\right)+\chi+\delta_x+6 \delta_y-\frac{\alpha^3}{12L_2^2}.\]
Also, given the fact that
\[\Phi(p_k)\leq \phi_{p_k}+\delta_y=\min\{\phi_{x^1},\phi_{x^2}\}+\delta_y\leq \min\{\Phi(x^1),\Phi(x^2)\}+2\delta_y,\]
we know
\(\Phi(p_k)\leq\Phi\left(x_{k, 0}\right)+\chi+\delta_x+8 \delta_y-\frac{\alpha^3}{12L_2^2}.\)
By additionally assuming $\chi+\delta_x+8 \delta_y\leq \frac{\alpha^3}{72L_2^2}$, we have
\(\Phi\left(p_k\right)-\Phi\left(p_{k-1}\right)=\Phi\left(p_k\right)-\Phi\left(x_{k, 0}\right) \leq-\frac{\alpha^3}{72 L_2^2}.\)
As a result, we get the following descent guarantee by exploiting negative curvature.
\begin{lemma}\label{lemma:Exploit-NC-Pair}
Suppose the algorithmic parameters are chosen so that $\chi+\delta_x+8 \delta_y\leq \frac{\alpha^3}{72L_2^2}$, $\chi + \delta_x + 4\delta_y\leq\frac{\alpha^3}{32L_2^2}$, $\delta_x\ge 2\delta_y$ and $\ell\sqrt{\frac{2 \delta_x}{\gamma+2 \alpha}} \geq 2 \Delta_y$.  In the $k$-th epoch of the Algorithm \ref{alg:main}, suppose Line 16 is executed. We have \(\Phi\left(p_k\right)-\Phi\left(p_{k-1}\right)\leq-\frac{\alpha^3}{72 L_2^2}.\)
\end{lemma}
\noindent We summarize all the requirements of parameter settings as follows

\begin{eqnarray}
\label{param}
    && \,\qquad\,\, \chi+\delta_x+\delta_y\overset{\rm\tiny Lm \, \ref{lemma:epln}}{\leq}\frac{\epsilon^2}{3200\gamma}, \qquad\, \chi + \delta_x + 4\delta_y\overset{\rm\tiny Lms \, \ref{lemma:Flag-5},\, \ref{lemma:Exploit-NC-Pair}}{\leq} \min\left\{\frac{\alpha^3}{32L_2^2},\frac{\epsilon^2}{50 \alpha}\right\},\nonumber\\
    &&\,\qquad\,\, \chi\!+\!\delta_x\!+\!8 \delta_y\!\overset{\rm\tiny Lm \, \ref{lemma:Exploit-NC-Pair}}{\leq} \frac{\alpha^3}{72L_2^2},\chi\!+\!4 \delta_y \!\overset{\rm\tiny Lm \, \ref{lemma:Flag-3}}{\leq} \frac{\alpha^3}{72 L_2^2},2\delta_y\!\overset{\rm\tiny Lm \, \ref{lemma:Exploit-NC-Pair}}{\leq} \delta_x,  \Delta_y\!\overset{\rm\tiny Lm \, \ref{lemma:Flag-5}}{\leq} \frac{\epsilon}{4},\\ 
    &&\,\qquad\,\,\,\,\,\, \frac{2 \Delta_y}{\ell} \overset{\rm\tiny Lm \, \ref{lemma:Exploit-NC-Pair}}{\leq} \sqrt{\frac{2 \delta_x}{\gamma+2 \alpha}}\overset{\rm\tiny Lms \, \ref{lemma:Flag-5},\,\ref{lemma:Flag-3}}{\leq} \min\left\{\frac{\epsilon}{20\left(L_1+2 \gamma\right)},\frac{\alpha}{24 L_2}\right\},\nonumber
\end{eqnarray}
where $\chi = 6\sqrt{\kappa_x}\cdot\left(11\kappa_x\delta_x+(2\ell+L_1+\alpha)\sqrt{\frac{2\delta_x}{\gamma+2\alpha}}\cdot d\right)$, $\kappa_x = \gamma/\alpha$, and $d=\alpha/L_2$.

Combining Lemmas \ref{lemma:Flag-5}, \ref{lemma:Flag-3}, \ref{lemma:Flag-24}, and \ref{lemma:Exploit-NC-Pair}, we know that if the algorithm does not terminate, then the per-epoch descent will be 
\[\Phi\left(p_k\right)-\Phi\left(p_{k-1}\right)\leq -\min \left\{\frac{\epsilon^2}{50\alpha}, \frac{\alpha^3}{72L_2^2}\right\}.\]
The per-epoch descent guarantee implies the main theorem of this section.

\begin{theorem}\label{main:theorem}
Suppose Assumptions \ref{assu:Phi} and \ref{error} hold and suppose $\epsilon\leq \frac{\ell^2}{L_2}$. By choosing $\alpha=\sqrt{L_2\epsilon}$, $\gamma=\ell$,  
 $\delta_x = \min\left\{\frac{L_2^2\epsilon^4}{10^{10}\kappa_y^2\ell^2}, \frac{\epsilon^2}{10^6\kappa_x^{3/2}\ell}\right\}$, $\delta_y = \frac{\delta_x}{2}$, and $\Delta_y = \min\left\{\frac{L_2\epsilon^2/4}{10^5\kappa_y\sqrt{\ell}}, \frac{\epsilon/3}{10^3\kappa_x^{3/4}}\right\}$ so that \eqref{param} holds, then the total iteration complexity for finding $\epsilon$-stationary point is
\( \tilde{\mathcal{O}}\left(\frac{\left(\Phi\left(p_0\right)-\Phi^*\right) \ell^{1/2} L_2^{1/4}}{\epsilon^{7 / 4}}\right).\)
\end{theorem}
Before presenting the proof, one remark is that requiring $\epsilon\leq\ell^2/L_2$ is only for simplifying the expressions of $\delta_x,\delta_y$ and $\Delta_y$, removing it only affects the logarithmic factors in the final complexity bound.  
\begin{proof}
To guarantee the first 5 inequalities of \eqref{param}, it is sufficient to set $\delta_y=\delta_x/2$ and require
$\chi + 5\delta_x \leq \min\left\{\frac{\epsilon^2}{3200\gamma},\frac{\epsilon^2}{50 \alpha}, \frac{\alpha^3}{72L_2^2}\right\}$. The relationships $\gamma=\ell$ and $0<\epsilon\leq \frac{\ell^2}{L_2}$ simplify the above requirements to $\chi + 5\delta_x \leq  \frac{\epsilon^2}{3200\gamma}$. Recall the definition of $\chi$, it is sufficient to require:\vspace{-0.2cm}
\begin{equation} 
    (66\kappa_x^{1.5}+5)\delta_x \leq \frac{\epsilon^2}{6400\gamma}\qquad\mbox{and}\qquad 6\sqrt{\kappa_x}(2\ell+L_1+\alpha)\sqrt{\frac{2\delta_x}{\gamma+2\alpha}}\cdot\frac{\alpha}{L_2} \leq \frac{\epsilon^2}{6400\gamma}.\nonumber
\end{equation} 
Note that $\frac{L_2\cdot\epsilon^2}{6400\gamma(6\sqrt{\kappa_x}(2\ell+L_1+\alpha))\alpha}\leq\min\left\{\frac{\epsilon}{20\left(L_1+2 \gamma\right)},\frac{\alpha}{24 L_2}\right\}$ holds automatically due to the relationship between $L_2,\ell,\epsilon$; the last inequality of \eqref{param} is also guaranteed. Therefore, it suffices to choose $$\delta_x = \min\left\{\frac{L_2^2\epsilon^4}{10^{10}\kappa_y^2\ell^2}, \frac{\epsilon^2}{10^6\kappa_x^{3/2}\ell}\right\}, \quad \delta_y = \frac{\delta_x}{2}, \quad \Delta_y = \min\left\{\frac{L_2\epsilon^2/4}{10^5\kappa_y\sqrt{\ell}}, \frac{\epsilon/3}{10^3\kappa_x^{3/4}}\right\}.$$
Given these coefficients, \eqref{param} holds true. Combining Lemmas \ref{lemma:Flag-5}, \ref{lemma:Flag-3}, and \ref{lemma:Exploit-NC-Pair}, we know that if the algorithm does not terminate, then the per-epoch descent will be 
\[\Phi\left(p_k\right)-\Phi\left(p_{k-1}\right)\leq -\min \left\{\frac{\epsilon^2}{50\alpha}, \frac{\alpha^3}{72L_2^2}\right\}=-\min \left\{\frac{\epsilon^{3/2}}{50\sqrt{L_2}}, \frac{\epsilon^{3/2}}{72\sqrt{L_2}}\right\}=-\frac{\epsilon^{3/2}}{72\sqrt{L_2}}.\] 
Suppose the number of epochs is $K$. By telescoping the above inequality, we obtain
\[\Phi\left(p_0\right)-\Phi^*\geq \Phi\left(p_0\right)-\Phi\left(p_{K-1}\right)= \sum_{k=1}^{K-1}\left(\Phi\left(p_{k-1}\right)-\Phi\left(p_k\right)\right)\geq (K-1)\frac{\epsilon^{3/2}}{72\sqrt{L_2}}.\]
Hence we conclude
$K\leq 1+\frac{72\sqrt{L_2}}{\epsilon^{3/2}}\left(\Phi\left(p_0\right)-\Phi^*\right).$
Combined with Lemma \ref{lemma:epln}, the iterations number of each epoch is at most $1 + 6\sqrt{\frac{\ell}{\sqrt{L_2\epsilon}}}\cdot\max\!\left\{0,\log\!\left(\frac{3200\ell(\Phi(x_{k,0})-\Phi^*+2\delta_y)}{\epsilon^2}\right)\!\right\}$,  the total iteration complexity can be obtained. 
\end{proof}

\begin{remark}
Compared with CUPG \cite{carmon2017convex} that has an $\tilde{\mathcal{O}}\big((\Phi\left(p_0\right)-\Phi^*) L_1^{\frac{1}{2}} L_2^{\frac{1}{4}}\epsilon^{-7 / 4}\big)$ complexity, IAPUN replaces the $L_1^{\frac{1}{2}}$ dependence with an $\ell^{\frac{1}{2}}$, which can be much smaller in many situations, including the NC-SC minimax problems in this paper.  
\end{remark}

\subsection{Solving NC-SC minimax problems}
\label{subsec:NC-SC-deterministic}
In this section, we apply IAPUN to the NC-SC minimax problems, which requires us to provide the complexity for estimating the function and gradient of $\Phi(x)$, as well as the complexity for solving the APPA subproblem of IAPUN (Algorithm \ref{alg:main}, Line 8). \vspace{0.2cm}

\noindent\textbf{Approximate evaluation of  $\Phi$ and $\nabla \Phi$.}
In Algorithm \ref{alg:main} (Lines 14, 17), and Algorithm \ref{alg:Exploit-Ncvx} (Lines 4, 7) and Algorithm \ref{alg:Certify} (Lines 2, 10, 16), given any $x$, we need to compute estimators $\Phi_x$ and $g_x$ that satisfy Assumption \ref{error}. For any tolerance $\varepsilon>0$ and $x\in\mathbb{R}^{n}$, 
let $y_x^\varepsilon\in\mathbb{R}^{n'}$ be some point
such that $f(x,y_x^\varepsilon)\geq \max_y f(x,y) - \varepsilon$. By Danskin's Theorem, for properly chosen tolerances $\varepsilon_1,\varepsilon_2>0$, we may set
\begin{equation}
\label{estimators}
\Phi_x = f(x,y_x^{\varepsilon_1}) \quad\mbox{and}\quad
\,g_x = \nabla_x f(x,y_x^{\varepsilon_2}).
\end{equation}
Next, we characterize the $\tilde{O}\big(\sqrt{\kappa_y}\big)$ evaluation complexities in the following lemma. 
\begin{lemma}
\label{lemma:NC-SC-estimators}
Let us set $\varepsilon_1=\delta_y$ and $\varepsilon_2 = \frac{\mu\Delta_y^2}{2\ell^2}$ in \eqref{estimators}, then Assumption \ref{error} holds. If Nesterov's accelerated gradient method is adopted, then it takes $O\big(\sqrt{\kappa_y}\ln \frac{1}{\varepsilon_i}\big)$ first-order oracles to find $y_x^{\varepsilon_i}$, for each $i=1$  and $2$ respectively.
\end{lemma}
\begin{proof}
    First of all, for $\Phi_x = f(x,y_x^{\varepsilon_1})$, by definition, it is straightforward to verify that $|\Phi_x-\Phi(x)|\leq \delta_y$. For $g_x = \nabla_x f(x,y_x^{\varepsilon_2})$, note that 
    $$\frac{\mu}{2} \|y_x^{\varepsilon_2}-y^*(x)\|^2\leq \max_y f(x,y)-f(x,y_x^{\varepsilon_2})\leq \frac{\mu \Delta_y^2}{2\ell^2}.$$
    Note that Danskin's theorem indicates that $\nabla \Phi(x)=\nabla_x f(x,y^*(x)).$ Then we have 
    $$\|g_{x}-\nabla \Phi(x)\|=\|\nabla_x f(x,y_x^{\varepsilon_2})-\nabla_x f(x,y^*(x))\|\leq \ell \|y_x^{\varepsilon_2}-y^*(x)\|\le \Delta_y.$$
    Finally, the complexities of accelerated gradient method are direct results of \cite{nesterov2018lectures}.
\end{proof}

\noindent\textbf{Solving proximal point subproblems.}
In Algorithm \ref{alg:main} (Line 8) and  Algorithm \ref{alg:Certify} (Line 9, 15), we need to solve proximal point subproblems of the following form \vspace{-0.05cm}
\begin{equation}
\label{prob:PPA-sub}
\min_{x\in\mathcal{X}} \max_{y}\,\, \psi(x,y):=f(x,y) +\alpha\|x-p\|^2+ \gamma\|x-\tilde{x}\|^2,
\end{equation}
where $\cX$ is either $\RR^n$ or an L-2  ball. Note that $\psi$ is $(2\alpha+2\gamma-\ell)$-strongly convex in $x$ and  $\mu$-strongly concave in $y$,  with gradient Lipschitz  constant $2\alpha+2\gamma+\ell$. Given $\gamma \!=\! \ell \!\gg\! \alpha$, the condition number of $\psi$ in $x$ is $\frac{2\alpha+2\gamma+\ell}{2\alpha+2\gamma-\ell}\leq 3$. Hence, by \cite{lin2020near},  solving the proximal point subproblems requires $\tilde{\mathcal{O}}\big(\!\sqrt{3\kappa_y}\ln\delta_x^{-1}\big) = \tilde{\mathcal{O}}\big(\!\sqrt{\kappa_y}\big)$ first-order oracles. \vspace{0.1cm}

\begin{corollary}\label{theo:minimax}
Under Assumptions \ref{assu:minimax1}, \ref{assu:minimax2} and \ref{assu:minimax3}, using Algorithm \ref{alg:main} with proper parameters, the first-order oracle complexity for finding $\epsilon$-stationary point is
$$\tilde{\mathcal{O}}\left(\frac{\sqrt{\kappa_y}\left(\Phi\left(p_0\right)-\Phi^*\right) \ell^{1/2} L_2^{1/4}}{\epsilon^{7 / 4}}\right).\vspace{-0.05cm}$$
\end{corollary}
\begin{proof}
    First, Theorem \ref{main:theorem} indicates that at most $\tilde{\mathcal{O}}\big((\Phi\left(p_0\right)-\Phi^*) \ell^{1/2} L_2^{1/4}\epsilon^{-7 / 4}\big)$ iterations of IAPUN is needed. In each iteration of IAPUN, evaluating $\phi_x,g_x$ as well as approximately solving proximal point subproblems requires $\Tilde{O}(\sqrt{\kappa_y})$ first-order oracle evaluations. Combining both arguments leads to the reported overall complexity. 
\end{proof}

\subsection{Solving finite-sum NC-SC minimax problems}
As IAPUN works with inexact first-order oracle evaluations, it can naturally handle the finite-sum setting \vspace{-0.05cm}
\begin{equation}
\label{prob:finite-sum}
\min _{x} \max _{y} f(x, y):= \frac{1}{m} \sum_{i=1}^m f_i(x, y),\vspace{-0.05cm}
\end{equation} 
where the objective function satisfies the following smoothness assumption.

\begin{assumption}
\label{assumption:finite-sum}
For $\forall i=1,2,...,m$, the function  $f_i$ is $\ell$-smooth, $\mu$-strongly concave in $y$  and  nonconvex in $x$. %Let us still adopt the definition 
Still denoting $\Phi(x):=\max_y f(x,y)$, we assume $\Phi$ to be lower bounded and its Hessian is $L_2$-Lipschitz continuous. 
\end{assumption}
Note that Assumption \ref{assumption:finite-sum} implies Assumptions \ref{assu:minimax1}, \ref{assu:minimax2}, and \ref{assu:minimax3}, thus Lemma \ref{lemma:minimax} still applies. Several variance reduced first-order methods \cite{luo2020stochastic,zhang2021complexity} are proposed to solve this problem,  while \cite{zhang2021complexity} achieves the state-of-the-art   $\tilde{\mathcal{O}}\left(\left(m+m^{3/4}\sqrt{\kappa_y}\right)\ell\epsilon^{-2}\right)$ complexity. With second-order smoothness, we further improve this bound as follows.\vspace{0.1cm}

\noindent\textbf{Approximate evaluation of  $\Phi$ and $\nabla \Phi$.} Analogous to Section \ref{subsec:NC-SC-deterministic}, for $\forall\varepsilon>0$ and $\forall x\in\mathbb{R}^{n}$, 
let $y_x^\varepsilon\in\mathbb{R}^{n'}$ be a random point
s.t.  $\mathbb{E}[f(x,y_x^\varepsilon)]\geq \max_y f(x,y) - \varepsilon$. We can still adopt the estimators in \eqref{estimators}, with evaluation complexities described below.
\begin{lemma}
\label{lemma:finite-sum-estimators}
Given any failure probability $\delta\in(0,1)$, set $\varepsilon_1=\delta\delta_y$ and $\varepsilon_2 = \frac{\mu\delta\Delta_y^2}{2\ell^2}$ in \eqref{estimators}, then Assumption \ref{error} holds with probability at least $1-\delta$. Let one evaluation of $f_i$ or $\nabla f_i$ for any $i$ be one first-order oracle. If Katyusha Algorithm \cite{allen2017katyusha} is adopted, then it takes $O\big(m+\sqrt{m\kappa_y}\ln \frac{1}{\varepsilon_i}\big)$ first-order oracles to find $y_x^{\varepsilon_i}$, for $i=1,2$.
\end{lemma}
\begin{proof}
    By Markov's inequality, for $i=1$ and $2$, we have $$\mathrm{Prob}\Big(f(x,y_x^{\varepsilon_i})\leq \max_y f(x,y) - \varepsilon_i/\delta\Big) \leq \frac{\mathbb{E}[\max_y f(x,y)-f(x,y_x^{\varepsilon_i})]}{\varepsilon_i/\delta}\leq \delta.$$
    Then following the discussion of Lemma \ref{lemma:NC-SC-estimators}, we conclude that Assumption \ref{error} holds with probability $1-\delta$ for $\Phi_x$ and $g_x$ respectively. That is, $$\mathrm{Prob}\big(|\Phi_x-\Phi(x)|\leq \delta_y\big)\geq 1-\delta\quad\mbox{and}\quad \mathrm{Prob}\big(\|g_x-\nabla\Phi(x)\|\leq \Delta_y\big)\geq 1-\delta.$$
    Finally, the complexity of the Katyusha algorithm directly follows \cite{allen2017katyusha}.
\end{proof}
\vspace{0.1cm}

\noindent{\textbf{Solving proximal point subproblems.} Similar to Section \ref{subsec:NC-SC-deterministic}, IAPUN iteratively solves proximal point subproblem \eqref{prob:PPA-sub} with objective function $f$ taking the finite-sum form \eqref{prob:finite-sum}. Note that with $\gamma=\ell\gg\alpha$, the condition number of this subproblem is $\kappa_x'\!=\!\frac{2\alpha+2\gamma+\ell}{2\alpha+2\gamma-\ell}\leq 3$ for $x$ and $\kappa_y'=\frac{2\alpha+2\gamma+\ell}{\mu}\leq 5\kappa_y$ for $y$. 
For any tolerance $\varepsilon>0$, let $x_+^\varepsilon$ be a random point s.t.
$\mathbb{E}\big[\!\max_y \!\psi(x_+^\varepsilon,y) \!-\! \min_{x\in\mathcal{X}} \max_y\! \psi(x,y)\big]\!\leq\! \varepsilon$. Then it takes $$\tilde{O}\Big(\sqrt{m\big(\sqrt{m}+\kappa_x'\big)\big(\sqrt{m}+\kappa_y'\big)}\ln\frac{1}{\varepsilon}\Big) = \tilde{\mathcal{O}}\big(m+m^{3/4}\sqrt{\kappa_y}\big)$$
first-order oracles for the AL-SVRE algorithm  \cite[Corollary 2]{luo2021near} to find such an $x_+^\varepsilon$. Given any failure probability $\delta\in(0,1)$, similar to the discussion of Lemma \ref{lemma:finite-sum-estimators}, setting $\varepsilon=\delta\delta_x$ and returning $x_+^{\delta\delta_x}$ guarantees the $\delta_x$ sub-optimality gap with probability at least $1-\delta$, in Algorithm \ref{alg:main} (Line 8) and  Algorithm \ref{alg:Certify} (Line 9, 15).

Note that each (inner) iteration of IAPUN calls at most 8 evaluations of $\Phi_x$, $g_x$, and proximal point subproblem \eqref{prob:PPA-sub}. Hence, let $K=\tilde{\mathcal{O}}\big(\!\left(\Phi\left(p_0\right)\!-\!\Phi^*\right) \ell^{1/2} L_2^{1/4}\epsilon^{-7 / 4}\big)$ be the total iteration number to find an $\epsilon$-stationary point (cf.\ Theorem \ref{main:theorem}), and we set $\delta = \frac{q}{8K}$ with $q\in(0,1)$ to be the failure probability of each approximate evaluation described above, then the union bound gives the following corollary. 
\begin{corollary}
Suppose Assumption \ref{assumption:finite-sum} hold. For $\forall q\in(0,1)$, by evaluating $\Phi_x$ and $g_x$ with Katyusha Algorithm \cite{allen2017katyusha} and evaluating the proximal point subproblem with AL-SVRE algorithm \cite{luo2021near}, with properly chosen parameters, it takes IAPUN 
\(
\tilde{\mathcal{O}}\big(\big(m+m^{3/4}\sqrt{\kappa_y}\big)\ell^{1/2} L_2^{1/4}\epsilon^{-7 / 4}\big)
\)
first-order oracles to find an $\epsilon$-stationary point 
with probability at least $1-q$. In particular, the dependence on $q$ is only $O(\log 1/q)$.
\end{corollary}

\section{Lower bound for NC-SC problems}\label{lower}
In this section, we derive the lower complexity bound for first-order methods to find $\epsilon$-stationary points of NC-SC problems satisfying Assumptions \ref{assu:minimax1}, \ref{assu:minimax2}, and \ref{assu:minimax3}. To simplify the notation, we denote $\Delta=\Phi(x_0)-\min_x \Phi(x)$. First, let us formally state the key definitions. 

\begin{definition}[First-order oracle]
For each query on point $(x,y)$, a first-order oracle returns the function and gradient values $(f(x,y), \nabla_x f(x,y),\nabla_y f(x,y))$.
\end{definition}

\begin{definition}[First-order algorithm] 
Let $\{(x_t,y_t)\}$ be the sequence of queries by some first-order
 algorithm. Then, the $(t + 1)$-th iterate $(x_{t+1},y_{t+1})$ satisfies
\[\operatorname{supp}(x_{t+1},y_{t+1}) \subset \bigcup_{0 \leq i \leq t}\left(\operatorname{supp}(x_i,y_i) \cup \operatorname{supp}(\nabla_x f(x,y),\nabla_y f(x,y))\right).\]
\end{definition}

Next, we introduce the concept of zero-chain \cite{nesterov2018lectures}, which is a common structure for designing hard instances and deriving lower bounds.
\begin{definition}
A function $f$ is a zero-chain if for any $z$ satisfying $\operatorname{supp}\{z\} \subseteq\{1,..., i-1\}$, then
\(\operatorname{supp}\{\nabla f(z)\} \subseteq\{1,..., i\}.\)
\end{definition}
Observe that if the sequence $\{(x_t,y_t)\}$ is generated by some first-order algorithm, with $(x_0,y_0)=(\textbf{0},\textbf{0})$, and if the function $f$ is a zero-chain, then
$\operatorname{supp}\{(x_t,y_t)\}\subseteq\{1,..., t\}.$ To derive a lower bound for NC-SC problems, let us first review the following nonconvex function constructed by \cite{carmon2021lower}:
\begin{eqnarray}\label{defn:NC-instance}
    \bar{f}^{nc}_{T, \nu}(x):&=&\frac{\sqrt{\nu}}{2}\left(x^1-1\right)^2+\frac{1}{2}\sum_{i=1}^{2T} (x^{i}-x^{i+1})^2+\nu \sum_{i=1}^{2T} \Upsilon\left(x^i\right),
\end{eqnarray}    
where the nonconvex function $\Upsilon$ is defined by 
\(\Upsilon(x):=120 \int_1^x \frac{t^2(t-1)}{1+t^2} \mathrm{d} t.\)
By \cite{carmon2021lower}, the functions $\Upsilon(x)$ and $\bar{f}^{nc}_{T, \nu}(x)$ satisfy the following properties.

\begin{lemma}\cite[Lemma 2(v)]{carmon2021lower}\label{lemma:hardinstance} The function $\Upsilon(x)$ is $\ell_1$-smooth and $\nabla^2\Upsilon(x)$ is $\ell_2$-Lipschitz continuous for some absolute constants $\ell_1,\ell_2>0$.
\end{lemma}

\begin{lemma}\cite[Lemma 4(i)]{carmon2021lower}\label{fun_bound} 
We have
\(\bar{f}^{nc}_{T,\nu}(0)-\inf _x \bar{f}^{nc}_{T,\nu}(x) \leq \frac{\sqrt{\nu}}{2}+20 \nu T.\)
\end{lemma}
\begin{lemma}\cite[Lemma 3]{carmon2021lower}\label{grad_lower} 
Let $\nu\le 1$. For any $x\in \mathbb{R}^{2T+1}$ with $x^{2T}=x^{2T+1}=0$, we have 
\(\|\nabla\bar{f}^{nc}_{T,\nu}(x) \|>\nu^{3/4}/4.\)
\end{lemma}

\noindent Now, let us consider the following hard instance from the NC-SC problem class:
\begin{eqnarray}
\label{defn:NC-SC-instance-0}
\bar{f}_{T,\nu}(x;\bar{y}) &=& \frac{\sqrt{\nu}}{2}\left(x^1-1\right)^2+\frac{1}{2} \sum_{i=1}^T\left(x^{2i-1}-x^{2i}\right)^2 \\
& & +\nu \sum_{i=1}^{2T} \Upsilon\left(x^i\right)+\sum_{i=1}^{T} h\left(x^{2i}, x^{2i+1} ; \bar{y}^{(i)}\right)+c\sum_{i=2}^{2T+1} (x^{i})^2\nonumber
\end{eqnarray}
where $x\in\mathbb{R}^{2T+1}$ and $\bar{y}^{(i)}\in\mathbb{R}^{n}$ for $i=1,\cdots,T$, namely, $\bar{y}\in\mathbb{R}^{Tn}$. The parameter $c$ will be specified later. The function $h:\mathbb{R}\times\mathbb{R}\times\mathbb{R}^{n}\to\mathbb{R}$ is defined as 
\begin{eqnarray}
\label{defn:NC-SC-instance-1}
    h(x,z;y)&:=&-\frac{1}{2} \sum_{i=1}^{n-1}\left(y^i-y^{i+1}\right)^2-\frac{1}{2 n^2}\|y\|_2^2+\sqrt{\frac{C}{n}}\left(x y^1- z y^n\right)\nonumber\\
    &=&-\frac{1}{2} y^{\top}\left(\frac{1}{n^2} I_n+A\right) y+\sqrt{\frac{C}{n}} b_{x, z}^{\top} y
\end{eqnarray}
where $b_{x, z}:=x e_1- z e_n$ and 
\[A:=\begin{pmatrix}
1 & -1 & & & \\
-1 & 2 & -1 & & \\
& -1 & \ddots & \ddots & \\
& & \ddots & 2 & -1 \\
& & & -1 & 1
\end{pmatrix}\]  
is a positive semidefinite matrix and $\|A\|_2\leq 4$. Consequently, $h(x,z;y)$ is $1/n^2$-strongly concave in $y$.
By straight computation, it is not hard to observe that
\[h^m(x, z):=\max _{y \in \mathbb{R}^n} h(x, z ; y)=\frac{C}{2 n} b_{x, z}^{\top}\left(\frac{1}{n^2} I_n+A\right)^{-1} b_{x, z}.\]

\begin{lemma}\cite[Lemma 5]{li2021complexity}
Suppose $n\ge 10$, then there exist $n$-independent constants $f_1>d_1>0$, $f_2>d_2>0$, and $n$-dependent constants $a_1\in [d_1,f_1]$, $a_2\in[d_2,f_2]$, such that 
$h^m(x, z)=C\left(\frac{a_1}{2} x^2-a_2 x z+\frac{a_1}{2} z^2\right)$. Here, $n$-dependent constants refer to the constants that depend on $n$, and $n$-independent constants refer to the constants that don't depend on $n$.
\end{lemma}
\noindent Now let us set $C=1/a_2$ in \eqref{defn:NC-SC-instance-1} and $c=C(a_2-a_1)/2$ in \eqref{defn:NC-SC-instance-0}, and substitute the partial maximization of the $\bar{y}$ variables for the $h(\cdot)$, we have 
\begin{eqnarray}
    \bar{f}_{T,\nu}^m(x):&=& \max_{\bar{y} \in \mathbb{R}^n}  \bar{f}_{T,\nu}(x;\bar{y})=\frac{\sqrt{\nu}}{2}\left(x^1-1\right)^2+\frac{1}{2}\sum_{i=1}^{2T} (x^{i}-x^{i+1})^2+\nu \sum_{i=1}^{2T} \Upsilon\left(x^i\right).\nonumber
\end{eqnarray}    
That is, $\bar{f}_{T,\nu}^m(x)=\bar{f}^{nc}_{T, \nu}(x)$ defined by \eqref{defn:NC-instance}.
According to Lemma \ref{lemma:hardinstance} and the fact that $\|\frac{1}{n^2} I_n+A\|\le 5$, if we choose $\nu\leq 1$, the function $\bar{f}_{T,\nu}(x;\bar{y})$ is $\bar{\ell}_1$-smooth for some absolute constant $\bar{\ell}_1$ and  $\nabla^2\bar{f}_{T,\nu}^m(x)$ is $\nu \bar{\ell}_2$-Lipschitz continuous for some absolute constant $\bar{\ell}_2$. Note that by absolute constant, we mean a constant real number that does not depend on any other parameters. Then we apply the following scaling to manipulate the smoothness constants of a problem:
\begin{equation}
    \label{defn:NC-SC-instance-2}
    f_{T,\nu}(x;\bar{y})=\frac{\ell\lambda^2}{\bar{\ell}_1}\bar{f}_{T,\nu}\left(\frac{x}{\lambda};\frac{\bar{y}}{\lambda}\right).
\end{equation} 
Then by setting $\nu=\frac{L_2\bar{\ell}_1\lambda}{\bar{\ell}_2 \ell}\leq 1$, the function $f_{T,\nu}(x;\bar{y})$ defined by \eqref{defn:NC-SC-instance-2} is $\ell$-smooth and $\nabla^2 f_{T,\nu}^m(x)$ is $L_2$-Lipschitz continuous, we will show $\nu\leq 1$ later. By setting $n=\Theta\left(\sqrt{\kappa_y/\bar{\ell}_1}\right)$, we know $f_{T,\nu}(x;\cdot)$ is $\mu$-strongly concave. Finally, we set 
$$\lambda= \left(\frac{4\epsilon}{\left(\ell/\bar{\ell}_1\right)^{1/4}\left(L_2/\bar{\ell}_2\right)^{3/4}}\right)^{4/7}\qquad\quad\mbox{and}\qquad\quad T=\left\lfloor\frac{\Delta-\frac{\sqrt{\nu}\ell\lambda^2}{2\bar{\ell}_1}}{20\frac{\nu\ell\lambda^2}{2\bar{\ell}_1}}\right\rfloor.$$
By the definition of $T$ and Lemma \ref{fun_bound}, we know $f_{T,\nu}(0)-\inf_x f_{T,\nu}(x)\le \Delta$.
We focus on small enough $\epsilon$ such that $\epsilon\le \frac{1}{4}\left(\frac{\ell}{\bar{\ell}_1}\right)^{2}\left(\frac{L_2}{\bar{\ell}_2}\right)^{-1}$. We can observe that 
\[\nu=\frac{L_2\bar{\ell}_1\lambda}{\bar{\ell}_2 \ell}=\left(\frac{4\epsilon}{\left(\ell/\bar{\ell}_1\right)^{2}\bar{\ell}_2/L_2}\right)^{4/7}\le 1.\]

Combined with Lemma \ref{fun_bound}, we know $f_{T,\nu}(x;\bar{y})$ satisfies Assumptions \ref{assu:minimax1}, \ref{assu:minimax2}, and \ref{assu:minimax3}. Next, we will show at least an $n(T-1)$ complexity is needed to obtain an $\epsilon$-stationary point with the idea of zero-chain. Given $\Upsilon'(0)=0$, we observe the nonzero chain expands in the following way.

\begin{observation}\label{obser:zero}
For $\forall n, T\in \mathbb{N}_+$, $\nu>0$, the zero chain of $\bar{f}_{T,\nu}(x;\bar{y})$ expands as: $x^1,x^2$, $y^{(1)1},y^{(1)2},...,y^{(1)n}$,   $x^3,x^4,y^{(2)1},y^{(2)2},...,y^{(2)n}$, $...,x^{2T+1}$, where we denote $y^{(i)}=(y^{(i)1},y^{(i)2},\cdots,y^{(i)n})^T$. We visualize this chain as follows: \vspace{-0.2cm}
\begin{figure}[H]
    \centering
    \includegraphics[height=1.1cm]{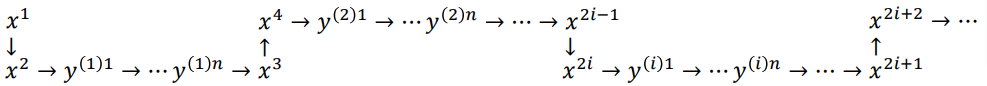}
    \vspace{-0.3cm}
\end{figure}
\noindent In the above chain, the vertical expansion is introduced by the $(x^{2i-1}-x^i)^2$ terms while the horizontal expansions are introduced by the $h(x^{2i},x^{2i+1};y^{(i)})$ terms.
\end{observation}

\noindent By Observation \ref{obser:zero}, for any $t\le n(T-1)$, the iteration $x_t$ must satisfy $x_t^{2T}=x_t^{2T+1}=0$. Then Lemma \ref{grad_lower} immediately indicates that \vspace{-0.1cm}
\[\|\nabla f_{T,\nu}^m(x_t)\|\ge \frac{\ell\lambda}{\bar{\ell}_1}\nu^{3/4}/4=\frac{1}{4}\left(\frac{\ell}{\bar{\ell}_1}\right)^{1/4}\left(\frac{L_2}{\bar{\ell}_2}\right)^{3/4}\lambda^{7/4}=\epsilon.\]
To satisfy $\|\nabla f_{T,\nu}^m(x)\|\le \epsilon$, at least $n(T-1)$ gradient evaluation is needed. Thus, we only need to give a lower bound on $T$, which can be divided into two cases.

\noindent\textbf{Case 1.} If $\Delta< \frac{\sqrt{\nu}\ell\lambda^2}{\bar{\ell}_1}$, then we have \vspace{-0.1cm}
\[\Delta< \frac{\sqrt{\nu}\ell\lambda^2}{\bar{\ell}_1}= 4^{10/7} \epsilon^{10/7}\left(\frac{\ell}{\bar{\ell}_1}\right)^{\!1/7}\!\left(\frac{L_2}{\bar{\ell}_2}\right)^{\!-4/7}\Longrightarrow\,\,\,\,\, \epsilon > \frac{1}{4}\Delta^{7/10}\left(\frac{\ell}{\bar{\ell}_1}\right)^{-1/10}\left(\frac{L_2}{\bar{\ell}_2}\right)^{2/5}\]
However, because we only focus on small enough $\epsilon$, we will consider the other case.

\noindent\textbf{Case 2.} If $\Delta\ge\frac{\sqrt{\nu}\ell\lambda^2}{\bar{\ell}_1}$, then we have \vspace{-0.1cm}
\[2T+1\ge \frac{\Delta}{40\frac{\nu\ell\lambda^2}{2\bar{\ell}_1}}=\frac{\Delta}{20\frac{L_2\lambda^3}{\bar{\ell}_2}}=\frac{\Delta \ell^{3/7}L_2^{2/7}}{20\cdot 4^{12/7}\cdot\bar{\ell}_1^{3/7}\bar{\ell}_2^{2/7}\epsilon^{12/7}}.\]
Therefore, to find some $x$ s.t. $\|\nabla f_{T,\nu}^m(x)\|\le \epsilon$ for sufficiently small $\epsilon$, one needs at least
\(t\ge \Omega\left(\Delta\sqrt{\kappa_y}\ell^{\frac{3}{7}}L_2^{\frac{2}{7}} \epsilon^{-12 / 7}\right)\) gradient evaluations, as summarized in Theorem \ref{theorem:lower-bound}.

\begin{theorem}
\label{theorem:lower-bound}
For any $\Delta,\ell,L_2,\mu,\epsilon>0$, such that $\kappa_y=\ell/\mu_y\ge 1$, and $$\epsilon \le \min\left\{\frac{1}{4}\left(\frac{\ell}{\bar{\ell}_1}\right)^{2}\left(\frac{L_2}{\bar{\ell}_2}\right)^{-1},\frac{1}{4}\Delta^{7/10}\left(\frac{\ell}{\bar{\ell}_1}\right)^{-1/10}\left(\frac{L_2}{\bar{\ell}_2}\right)^{2/5}\right\},$$
there exists an NC-SC saddle point problem of form \eqref{prob:minimax} that satisfies the following properties: (i). The objective function $f$ is $\ell$-smooth and $\mu$-strongly concave in $y$; (ii) Let $\Phi(x):=\max_y f(x,y)$, then $\Phi(x)$ has $L_2$-Lipschitz Hessian and $\Phi(0)-\min_{x}\Phi(x)\leq \Delta$. To find some $\bar x$ s.t. $\|\nabla \Phi(\bar x)\|\le \epsilon$ for this NC-SC problem, any first-order algorithm will need at least $t\ge \Omega\left(\sqrt{\kappa_y}\Delta\ell^{3/7}L_2^{2/7} \epsilon^{-12 / 7}\right)$ gradient evaluations. 
\end{theorem}
Due to the construction of hard instance, the lower bound provided by Theorem \ref{theorem:lower-bound} is only effective for small enough accuracy that satisfies $\epsilon \leq O(\ell^2/L_2)$. In this accuracy regime, both of our lower (Theorem \ref{theorem:lower-bound}) and upper (Theorem \ref{theo:minimax}) bounds improve the existing lower and upper bounds of first-order methods without exploiting second-order smoothness \cite{zhang2021complexity,li2021complexity,lin2020near}. As there is only a negligible $\tilde{O}\big(\!(\frac{\ell^2}{L_2\epsilon})^{\frac{1}{28}}\big)$ gap between our lower and upper bounds, both of them are close to optimal.

\section{Further applications of IAPUN}\label{application}
In this section, we will introduce the further application of IAPUN to the linearly constrained minimax problems and the bilevel optimization problems.

\subsection{Minimax problems with coupled linear equality constraints}
In \cite{tsaknakis2021minimax,zhang2023primal,dai2022optimality}, a class of constrained minimax problems has been introduced where the primal and dual decision variables may subject to additional linear constraints. 
In this subsection, we focus on the linear equality constrained with coupling constraints of the following form:
\[\min _{x\in\mathbb{R}^{n}} \max _{y\in\mathbb{R}^{n'}: Ax+By=c} f(x, y),\]
where $A\in\mathbb{R}^{m\times n}$ and $B\in\mathbb{R}^{m\times n'}$. 
Same as \cite{tsaknakis2021minimax}, we impose the following assumption.
\begin{assumption}\label{assu:cminimax}
$f(x, y)$ is twice-continuous differentiable and $\ell$-smooth and $\mu$-strongly concave. For every $x$, there exists a vector $y$ such that $Ax + By -c=0$. 
\end{assumption}
Let the Lagrangian of the problem be
\[L(x,y,\lambda)=f(x,y)-\lambda^\top(Ax+By-c).\]
Given any $x$, since the existence of $y$ that satisfies the equation constraint, Slater’s condition is satisfied. Therefore, strong duality holds for the inner maximization:
\[\max _{y: Ax+By=c} f(x, y)=\max _{y} \min_\lambda L(x, y,\lambda)= \min_\lambda\max _{y} L(x, y,\lambda),\]
which further implies 
\[\min _x \max _{y: Ax+By=c} f(x, y)=\min _x\min_\lambda\max _{y} L(x, y,\lambda)=\min_{(x,\lambda)}\max _{y} L(x, y,\lambda).\]
Thus, the original problem is equivalent to an unconstrained minimax problem. Define $\Phi(x,\lambda)=\max _{y} L(x, y,\lambda)$ and $y^*(x,\lambda)=\mathop{\mathrm{argmax}}_{y} L(x, y,\lambda)$. 
According to Danskin's theorem, we obtain the gradient of function $\Phi(x,\lambda)$:
\[\nabla \Phi(x,\lambda)=\begin{pmatrix}
\nabla f(x,y^*(x,\lambda))-A^\top\lambda\\
-(Ax+By^*(x,\lambda)-c)
\end{pmatrix}.
\]
Consequently, we introduce the following notion of $\epsilon$-stationarity. 
\begin{definition}
We call $x$  an $\epsilon$-stationary point if there exists $\lambda$ such that
\[\|\nabla f(x,y^*(x,\lambda))-A^\top\lambda\|<\epsilon \mbox{ and }
\|Ax+By^*(x,\lambda)-c\|<\epsilon.\]
\end{definition}
By direct computation, any $\epsilon$-stationary point of $\Phi$ will be a $\Theta(\epsilon)$-stationary point of the original problem. According to Assumption \ref{assu:cminimax}, $L(x,y,\lambda)$ is $\mu$-strongly concave in $y$ and $\ell'$-smooth in $(x,y,\lambda)$ with $\ell'=\sqrt{4\ell^2+2\max\{\|A\|^2,\|B\|^2\}}$, indicating the validity of Assumption \ref{assu:minimax2}. Due to the twice differentiability of $f$, we also assume that Assumption \ref{assu:minimax3} holds with $\nabla^2\Phi$ being $L_2$-Lipschitz continuous. If a good estimation of $L_2$ is hard to obtain, we may adopt a pessimistic upper bound $L_2\leq \big(1+\frac{\ell'}{\mu}\big)^3\ell_2$, where $\ell_2$ is the Lipschitz constant of $\nabla^2 f$. Hence we have the following result. 

\begin{corollary} 
Suppose Assumption \ref{assu:cminimax} holds and  $L(x,y,\lambda), \Phi(x,\lambda)$ satisfy Assumptions \ref{assu:minimax2} and \ref{assu:minimax3}. Then IAPUN with properly chosen parameters can find an $O(\epsilon)$-stationary point with
\(
\tilde{\mathcal{O}}\left(\frac{\sqrt{\kappa_y}\Delta (\ell')^{1/2} L_2^{1/4}}{\epsilon^{7 / 4}}\right)
\) complexity, where $\kappa_y = \ell'/\mu$.
\end{corollary}

\subsection{Bilevel Optimization} 
In this section, we consider the unconstrained bilevel optimization problem of the following form
\begin{equation}
    \label{prob:bilevel}
    \min _{x \in \mathbb{R}^n} \Phi(x):=f\left(x, y^*(x)\right), \quad \text { where } y^*(x)=\mathop{\mathrm{argmin}}_{y \in \mathbb{R}^{n'}} g(x, y),
\end{equation}
where $f$ is a nonconvex function with uniformly bounded partial gradient $\|\nabla_yf(x,\cdot)\|$, and $g$ is strongly convex in $y$. %Under such a NC-SC setting, with $\|\nabla_yf(x,\cdot)\|$ uniformly bounded above, \cite{ghadimi2018approximation} derived an algorithm that finds an $\epsilon$-stationary point of problem \eqref{prob:bilevel} with $\tilde{O}(\epsilon^{-2.5})$ complexity, which is further improved to $\tilde{O}(\epsilon^{-2})$ by \cite{ji2021bilevel}. 
In this paper, we will show an $\tilde{O}(\epsilon^{-7/4})$ complexity by exploiting the additional second-order smoothness in $\Phi$. Formally, we consider the bilevel optimization problems under the following standard assumptions. 

\begin{assumption}\label{assu:bif}
The function $f$ is $L_f$-smooth in $(x,y)$ and $C$-Lipschitz continuous in $y$. The function $g$ is $L_g$-smooth in $(x,y)$ and $\mu$-strongly convex in $y$. Moreover, we assume $\nabla^2g$ to be $\rho_g$-Lipschitz continuous and $\nabla^2\Phi$ to be $L_2$-Lipschitz continuous. 
\end{assumption}
Note that by  \cite{foo2007efficient,ji2021lower}, the gradient $\nabla\Phi$ should be evaluated as \vspace{-0.05cm}
\[\nabla \Phi(x)=\nabla_x f\left(x, y^*(x)\right)-\nabla_{xy}^2 g\left(x, y^*(x)\right)\left[\nabla_y^2 g\left(x, y^*(x)\right)\right]^{-1} \nabla_y f\left(x, y^*(x)\right).\vspace{-0.05cm}\]
Consequently, $\Phi$ is $L_1$-smooth with $L_1=\big(L_f + \frac{C\rho_g}{\mu}\big)(1+\kappa_y)^2$ and $\kappa_y = \frac{L_g}{\mu}$. 
Also note that $\nabla \Phi(x)$ relies on both the gradient of $f$ and the Hessian of $g$.  \cite{ghadimi2018approximation} suggests the quantities $\mathrm{GC}(f,\epsilon)$, $\mathrm{GC}(g,\epsilon)$, and $\mathrm{HC}(g,\epsilon)$ to characterize the complexity, which is defined as the total gradient evaluations of $f$, the total gradient evaluations of $g$, and the total Hessian evaluations of $g$ for finding an $\epsilon$-stationary point of $\Phi$, respectively. \vspace{0.15cm}

\noindent\textbf{Approximate evaluation of $\Phi$ and $\nabla\Phi$.} For any tolerance $\varepsilon>0$ and $x\in\mathbb{R}^n$, let $y_x^\varepsilon$ be some point s.t. $\|y_x^\varepsilon-y^*(x)\|\leq\varepsilon$. For properly chosen $\varepsilon_1,\varepsilon_2>0$, we can set
\begin{equation}
\label{estimator:bilevel}
\Phi_x\!=\!f(x,y_x^{\varepsilon_1}) \mbox{ and } g_x \!=\! \nabla_{\!x} f\left(x, y_x^{\varepsilon_2}\right)\!-\!\nabla_{\!xy}^2 g\!\left(x, y_x^{\varepsilon_2}\right)\!\left[\nabla_y^2 g\left(x, y_x^{\varepsilon_2}\right)\right]^{-1} \!\nabla_{\!y} f\!\left(x, y_x^{\varepsilon_2}\right).
\end{equation}
Their evaluation complexities are summarized in Lemma \ref{error:bilevel}, whose proof is omitted. 
\begin{lemma}
\label{error:bilevel}
Let us set $\varepsilon_1=\frac{\delta_y}{C}$ and $\varepsilon_2 = \frac{\mu\Delta_y}{(1+\kappa_y)(L_f\mu + C\rho_g)}$ in \eqref{estimator:bilevel}, then Assumption \ref{error} holds. If Nesterov's accelerated gradient method is adopted, then it takes $O\big(\sqrt{\kappa_y}\ln \frac{1}{\varepsilon_i}\big)$ first-order oracles to find $y_x^{\varepsilon_i}$, for each $i=1$  and $2$ respectively.\vspace{0.15cm}
\end{lemma}

\noindent\textbf{Solving the proximal point subproblems.}  In Algorithm \ref{alg:main} (Line 8) and  Algorithm \ref{alg:Certify} (Line 9, 15), we need to solve proximal point subproblems of the following form
$$\min_{x\in\mathcal{X}} \,\, \psi(x):=f(x,y^*(x)) +\alpha\|x-p\|^2 + \gamma\|x-\tilde{x}\|^2$$
with $y^*(x)$ defined by \eqref{prob:bilevel}. Note that the function $\psi$ is $(2\alpha+2\gamma-L_1)$-strongly convex and $(2\alpha+2\gamma+L_1)$-smooth. With $\gamma=L_1\gg\alpha$, the condition number of $\psi$ satisfies $\kappa_x=\frac{2\alpha+2\gamma+L_1}{2\alpha+2\gamma-L_1}\leq3$. Consequently, solving this subproblem with ABA algorithm \cite{ghadimi2018approximation} takes both $\tilde{O}\big(\sqrt{\kappa_x}+\kappa_y)=\tilde{O}(\kappa_y)$ gradient evaluations of $f$ and Hessian evaluations $g$, and $\tilde{O}(\kappa_x+\kappa_y^2) = \tilde{O}(\kappa_y^2)$ gradient evaluations of $g$.

\begin{corollary}
Under Assumption \ref{assu:bif}, with appropriate parameters,  IAPUN takes \(
\tilde{\mathcal{O}}\big(\frac{\Delta L_1^{1/2} L_2^{1/4}}{\epsilon^{7 / 4}}\big)\) iterations to find an $\epsilon$-stationary point. Moreover, we have 
$$\mathrm{GC}(f, \epsilon) = \mathrm{HC}(g, \epsilon)=\tilde{\mathcal{O}}\left(\frac{\kappa_y\Delta L_1^{1/2} L_2^{1/4}}{\epsilon^{7 / 4}}\right)\quad \mbox{and}\quad \mathrm{GC}(g,\epsilon)=\tilde{\mathcal{O}}\left(\frac{\kappa_y^2\Delta L_1^{1/2} L_2^{1/4}}{\epsilon^{7 / 4}}\right).$$
\end{corollary}
This result improves the existing $\tilde{O}(\epsilon^{-2})$ complexity of \cite{ji2021bilevel,chen2023near} by a factor of $O(\epsilon^{-1/4})$. Moreover, two months after the release of our paper, \cite{yang2023accelerating} parallelly rediscovered the $\tilde{O}(\epsilon^{-7/4})$ complexity by leveraging the restarted AGD scheme \cite{li2023restarted}.

\section{Conclusions} \label{conclusion}
We propose a fast algorithm to find $\epsilon$-stationary points of NC-SC minimax problems. With additional second-order smoothness, we improve the complexity of first-order method from $\tilde{O}\big(\sqrt{\kappa_y}\ell\epsilon^{-2}\big)$ to $\tilde{O}\big(\sqrt{\kappa_y}\ell^{1/2}L_2^{1/4}\epsilon^{-7/4}\big)$. An lower bound of $\Omega\big(\sqrt{\kappa_y}\ell^{3/7}L_2^{2/7}\epsilon^{-12/7}\big)$ is constructed, with only an $\tilde{O}(\epsilon^{-1/28})$ gap compared to the upper bound. Hence a potential future direction is to figure out whether this gap is because the lower bound is not tight enough or IAPUN algorithm is sub-optimal. Moreover, because only approximate function values and gradients are required by our method, it can be further applied to the minimax problem with coupled linear constraints and the bilevel optimization problems, where we also obtain the $\tilde{O}\big(\epsilon^{-7/4}\big)$ complexities.

\newpage

\appendix
\section{Proof of Lemma \ref{lemma:Flag-5}}\label{proof:lemma:Flag-5}

\begin{proof}
Since $\text{FLAG}=5$, we know $p_k=w_k^x$.
\begin{equation}
\label{prop:Flag-5-1}
\|\nabla \Phi(p_k)\| \geq \|g_{p_k}\| - \|g_{p_k}-\nabla \Phi(p_k)\|  \ge  \frac{3\epsilon}{4} - \Delta_y\ge \frac{\epsilon}{2}.
\end{equation} 
Denote $w^{*x}_k = \mathop{\mathrm{argmin}}_{x}\hpk(x)+\gamma\|x-x_{k,t}\|^2$. By Line 15 of Algorithm \ref{alg:Certify}  and the $\gamma+2\alpha$-strong convexity of $\hpk(x)+\gamma\|x-x_{k,t}\|^2$, we can also derive the bound $\|w^x_k-w^{*x}_k\|\leq \sqrt{\frac{2\delta_x}{\gamma+2\alpha}}.$ By KKT-condition, $\nabla\hpk(w^{*x}_k)=-2\gamma(w^{*x}_k-x_{k,t})$. Consequently
\begin{eqnarray}
    \label{prop:Flag-5-2}
    \|\nabla \hpk(w^{x}_k)\| &\leq& \|\nabla\hpk(w^{*x}_k)\| + L_1\|w^{x}_k-w^{*x}_k\|\\
    & \leq & 2\gamma\|w^{x}_k-x_{k,t}\|+(L_1+2\gamma)\|w^{x}_k-w^{*x}_k\|\nonumber \\
    & \leq &  \frac{\epsilon}{20} + (L_1+2\gamma)\sqrt{\frac{2\delta_x}{\gamma+2\alpha}} %\nonumber\\
    \leq  \frac{\epsilon}{10}.\nonumber
\end{eqnarray}
Combining \eqref{prop:Flag-5-1} and \eqref{prop:Flag-5-2} yields
$2\alpha\|w_k^x-x_{k,0}\| = \|\nabla\hpk(w_k^x)-\nabla\Phi(w_k^x)\|\geq \frac{\epsilon}{2}-\frac{\epsilon}{10} = \frac{2\epsilon}{5}.$
Further combining with the fact that $\hat{\phi}_{k(x_{k,t})}\leq \hat{\phi}_{k(x_{k,0})}+\chi+2\delta_y$ which is the necessary condition for returning $\text{Flag} = 5$, we have 
\begin{eqnarray*}
    &&\Phi(w^k_x) \!-\! \Phi(x_{k,0})\\ 
    & = & \hpk(w^k_x) \!-\! \Phi(x_{k,0}) \!-\! \alpha\|w_k^x\!-\!x_{k,0}\|^2 \leq  \hpk(x_{k,t}) \!+\! \delta_x \!-\! \Phi(x_{k,0}) \!-\! \frac{\epsilon^2}{25\alpha}\\
    & \leq & \hat{\phi}_{k(x_{k,t})} \!+\! 2\delta_y \!+\! \delta_x \!-\! \hat{\phi}_{k(x_{k,0})} \!-\!\frac{\epsilon^2}{25\alpha} \leq  \!-\!\frac{\epsilon^2}{25\alpha} \!+\! \cE \!+\! \delta_x \!+\! 4\delta_y \leq  \!-\! \frac{\epsilon^2}{50\alpha}
\end{eqnarray*}
where the second row in the above inequality is due to Line 15 of Algorithm \ref{alg:Certify} and  $2\alpha\|w_k^x-x_{k,0}\| \!\geq\! \frac{2\epsilon}{5}$. Finally, substituting $p_k = w^x_k$ and $p_{k-1} = x_{k,0}$ into the above inequality proves this lemma.
\end{proof}\vspace{0.2cm}

\section{Proof of Lemma \ref{lemma:Flag-3}} \label{proof:lemma:Flag-3}
\begin{proof}
For ease of discussion, let us set $t=T_k$ and use $x_{k,t}$ to be the point computed in Line 8 of Algorithm \ref{alg:main}. For the point recomputed in Line 8 of Algorithm \ref{alg:Certify}, we will write it as $\hat{x}_{k,t}$. When $\text{Flag}=3$, we have $w^x_k=\hat{x}_{k,t}$. For the subproblems that return $x_{k,t}$ and $\hat{x}_{k,t}$, we define their exact solutions as 
\begin{equation*}
    \begin{cases}
        x_{k,t}^* \!=\! \mathop{\mathrm{argmin}}_{x\in\RR^n}  \hpk(x) \!+\! \gamma\|x\!-\! \txkm\|^2\\
        \hat{x}_{k,t}^* \!=\! \mathop{\mathrm{argmin}}_{x\in\cX_k} \hpk(x) \!+\! \gamma\|x\!-\!\txkm\|^2
    \end{cases}
\end{equation*} 
Similar to previous discussion, we have $\|x_{k,t}-x_{k,t}^*\|\leq\sqrt{\frac{2\delta_x}{\gamma+2\alpha}}$ and $\|\hat{x}_{k,t}-\hat{x}_{k,t}^*\|\leq\sqrt{\frac{2\delta_x}{\gamma+2\alpha}}$.
To prove this lemma, let us consider two cases. \vspace{0.2cm}\\
\textbf{Case 1:} $x^*_{k,t} \in\cX_k$. In this case, we have $x^*_{k,t} = \hat x^*_{k,t}$. Therefore, 
\begin{eqnarray}
\|\hat x_{k,t}-x_{k,0}\| &\geq& \|x_{k,t}-x_{k,0}\| - \|x_{k,t}- x_{k,t}^*\| - \|x_{k,t}^*- \hat x_{k,t}^*\|-\|\hat x_{k,t}^*- \hat x_{k,t}\|\nonumber\\
&\geq& \frac{\alpha}{4L_2} - 2\sqrt{\frac{2\delta_x}{\gamma+2\alpha}} \geq  \frac{\alpha}{6L_2}. 
\end{eqnarray} 
\textbf{Case 2:} $x^*_{k,t}\notin\cX_k$, namely, $\|x^*_{k,t}-x_{k,0}\|\geq \frac{\alpha}{4L_2}$. In this situation, let us argue that must  $\hat x^*_{k,t}$ stay on the boundary of the ball $\cX_k$ by contradiction. Suppose $\hat{x}^*_{k,t}\in\mathrm{int}(\cX_k)$, then there exists some $\omega = (1-\lambda) \hat x^*_{k,t} + \lambda x^*_{k,t}$ with $\lambda\in(0,1)$ and $\lambda$ is small enough such that $\omega\in\cX_k$. Denote $\phi(x)=\hpk(x)+\gamma\|x-\txkm\|^2$. By the $(\gamma+2\alpha)$-strong convexity of $\phi$, we have  
$$\phi(\omega)\leq (1-\lambda)\phi(\hat{x}_{k,t}^*) + \lambda\phi(x^*_{k,t})\leq \phi(\hat{x}_{k,t}^*) - \lambda\cdot\frac{\gamma+2\alpha}{2}\|\hat x^*_{k,t} -x^*_{k,t}\|^2<\phi(\hat x^*_{k,t}).$$
This contradicts the optimality of $\hat x^*_{k,t}$. Now that we know $\hat x^*_{k,t}$ lies on the boundary of $\cX_k$, then $\|\hat x^*_{k,t}-x_{k,0}\| = \frac{\alpha}{4L_2}$. Consequently, we have 
$$\|\hat{x}_{k,t}-x_{k,0}\|\geq \|\hat x^*_{k,t}-x_{k,0}\| - \|\hat{x}_{k,t}^*-\hat{x}_{k,t}\|\geq \frac{\alpha}{4L_2} - \sqrt{\frac{2\delta_x}{\gamma+2\alpha}} \geq  \frac{\alpha}{6L_2}. $$
In both cases, we have $\|\hat{x}_{k,t}-x_{k,0}\|\geq\frac{\alpha}{6L_2}$. Finally, using the inequality $\hat{\phi}_{k(\hat x_{k,t})}\leq \hat{\phi}_{k(x_{k,0})}+\chi+2\delta_y$, which is the necessary condition for returning $\text{Flag}=3$, we have  
\begin{eqnarray*}
		\Phi(\hat x_{k,t}) - \Phi(x_{k,0}) & = & \hpk(\hat x_{k,t}) - \hpk(x_{k,0}) - \alpha\|\hat x_{k,t} - x_{k,0}\|^2\\
		&\leq & \hat{\phi}_{k(\hat x_{k,t})} - \hat{\phi}_{k(x_{k,0})}+2\delta_y - \frac{\alpha^3}{36L_2^2}\\
		& \leq & -\frac{\alpha^3}{36L_2^2} + \cE  + 4\delta_y
		\leq  -\frac{\alpha^3}{72L_2^2}.
	\end{eqnarray*} 
	Substituting $p_{k} = \hat{x}_{k,t}$ and $p_{k-1} = x_{k,0}$ into the above inequality proves the Lemma.  
\end{proof}

\section{Proof of Lemma \ref{lemma:aapp-intermediate}}\label{Proof:lemma:aapp-intermediate}
\begin{proof}
First,  let us denote the exact solution to the following subproblem as $x_{k,t}^*$:
\begin{eqnarray}
    \label{lm:aapp-1}
    x_{k,t}^*& \,\,= & \mathop{\mathrm{argmin}}_{x\in\cX}\,\,  \hpk(x) + \gamma\|x-\txkm\|^2.\nonumber
\end{eqnarray} 
For $t\leq T_k-1$, $\cX = \RR^n$. For $t=T_k$, $\cX = \RR^n$ in case $\text{Flag} = 1,4,5$ and $\cX=\cX_k$ in case $\text{Flag} = 2,3$. In all cases, by Line 8 of Algorithm \ref{alg:main}, we have 
$$\delta_x\geq\hpk(x_{k,t})+\gamma\|x_{k,t}-\txkm\|^2-\hpk(x_{k,t}^*)-\gamma\|x_{k,t}^*-\txkm\|^2 \geq \frac{\gamma+2\alpha}{2}\|x_{k,t}-x_{k,t}^*\|^2$$
where we use the fact that $\hpk(x)+\gamma\|x-\txkm\|^2$ is $(\gamma+2\alpha)$-strongly convex because we choose $\gamma\geq \ell$. This further indicates that 
\begin{equation}
    \label{lm:aapp-0}
    \|x_{k,t}-x_{k,t}^*\|\leq \sqrt{\frac{2\delta_x}{\gamma+2\alpha}}\qquad\mbox{for}\qquad 1\leq t\leq T_k.
\end{equation}
Second, when $\cX = \RR^n$, the KKT condition indicates $\nabla \hpk(x_{k,t}^*) + 2\gamma(x_{k,t}^*-\txkm) = 0$. When $\cX = \cX_k$, the KKT condition indicates 
$$\langle\nabla \hpk(x_{k,t}^*) + 2\gamma(x_{k,t}^*-\txkm),\bar x-x_{k,t}^*\rangle \geq 0,\qquad\mbox{for}\qquad \bar x\in\cX_k.$$
Therefore, in both cases, for any $\bar x\in\cX_k$ we have 
\begin{eqnarray*}
    & & \Big\{\hpk(\bar x) + \gamma\|\bar x-\txkm\|^2\Big\} - \Big\{\hpk(x_{k,t}^*) + \gamma\|x_{k,t}^*-\txkm\|^2\Big\}\\
    & \geq & \hpk(\bar x) - \hpk(x_{k,t}^*) + \gamma\|\bar x-\txkm\|^2 - \gamma\|x_{k,t}^*-\txkm\|^2 \\
    & & - \big(\nabla \hpk(x_{k,t}^*) + 2\gamma(x_{k,t}^*-\txkm) \big)^\top(\bar x-x_{k,t}^*)\\
    & = & \hpk(\bar x) - \hpk(x_{k,t}^*) - \nabla \hpk(x_{k,t}^*)^{\!\top}\!(\bar x-x_{k,t}^*) + \gamma\|\bar x - x_{k,t}^*\|^2.
\end{eqnarray*} 
Rearranging the terms yields
\begin{eqnarray}
    \label{lm:aapp-2}
    \qquad\quad\hpk(\bar x) & \geq & \hpk(x_{k,t}^*) \!+\! \gamma\|x_{k,t}^*\!-\!\txkm\|^2 \!-\! \gamma\|\bar x\!-\!\txkm\|^2 \!+\! \big(\gamma\!+\!\frac{\alpha}{2}\big)\|\bar x \!-\! x_{k,t}^*\|^2 \\
    && + \big(\hpk(\bar x) - \hpk(x_{k,t}^*) - \nabla \hpk(x_{k,t}^*)^{\!\top}\!(\bar x-x_{k,t}^*) - \frac{\alpha}{2}\|\bar x  - x_{k,t}^*\|^2\big)\nonumber\\
    & \geq & \hpk(x_{k,t}) \!+\! \gamma\|x_{k,t}\!-\!\txkm\|^2\!-\!\delta_x \!-\! \gamma\|\bar x\!-\!\txkm\|^2 \!+\! \big(\gamma\!+\!\frac{\alpha}{2}\big)\|\bar x \!-\! x_{k,t}^*\|^2\nonumber \\
    && + \big(\hpk(\bar x) - \hpk(x_{k,t}^*) - \nabla \hpk(x_{k,t}^*)^{\!\top}\!(\bar x-x_{k,t}^*) - \frac{\alpha}{2}\|\bar x  - x_{k,t}^*\|^2\big).\nonumber
\end{eqnarray}
For the last term of \eqref{lm:aapp-2}, we have 
\begin{eqnarray}
    \label{lm:aapp-3}
    \qquad\quad&&\hpk(\bar x) - \hpk(x_{k,t}^*) - \nabla \hpk(x_{k,t}^*)^{\!\top}\!(\bar x-x_{k,t}^*) - \frac{\alpha}{2}\|\bar x  - x_{k,t}^*\|^2\\
    &=& \hpk(\bar x) - \hpk(x_{k,t}) - \nabla\hpk(x_{k,t})^\top (\bar x - x_{k,t}) - \frac{\alpha}{2}\|\bar x-x_{k,t}\|^2 \nonumber\\
    & &\!\! + \big(\hpk(x_{k,t}) \!-\! \hpk(x_{k,t}^*) \!-\! \nabla\hpk(x_{k,t}^*)^{\!\top}\! (x_{k,t}\!-\!x_{k,t}^*)\big)\!\!+\! \frac{\alpha}{2}\big(\|\bar x\!-\!x_{k,t}\|^2\!-\!\|\bar x\!-\!x_{k,t}^*\|^2\big) \nonumber\\
    & & \!\!+ \big(\nabla\hpk(x_{k,t})-\nabla\hpk(x_{k,t}^*)\big)^\top(\bar x - x_{k,t}).\nonumber
\end{eqnarray}
By Assumption \ref{assu:Phi}, we have 
$$\hpk(x_{k,t}) \!-\! \hpk(x_{k,t}^*) \!-\! \nabla\hpk(x_{k,t}^*)^{\!\top}\! (x_{k,t}\!-\!x_{k,t}^*)\geq -\frac{\ell-2\alpha}{2}\|x_{k,t}-x_{k,t}^*\|^2 $$
and 
\begin{eqnarray}
    &&\frac{\alpha}{2}\big(\|\bar x\!-\!x_{k,t}\|^2\!-\!\|\bar x\!-\!x_{k,t}^*\|^2\big)+ \big(\nabla\hpk(x_{k,t})-\nabla\hpk(x_{k,t}^*)\big)^\top(\bar x - x_{k,t})\nonumber\\
    & = & \big(\nabla\Phi(x_{k,t})\!-\!\nabla\Phi(x_{k,t}^*)\!+\! 2\alpha(x_{k,t}\!-\!x_{k,t}^*)\big)^{\!\top}\!(\bar x \!-\! x_{k,t}) \!+\! \frac{\alpha}{2}\big(\|\bar x\!-\!x_{k,t}\|^2\!-\!\|\bar x\!-\!x_{k,t}^*\|^2\big)\nonumber\\
    & = & \big(\nabla\Phi(x_{k,t})\!-\!\nabla\Phi(x_{k,t}^*)\big)^{\!\top}\!(\bar x \!-\! x_{k,t}) - \frac{\alpha}{2}\|x_{k,t}-x_{k,t}^*\|^2 + \alpha(\bar x-x_{k,t})^\top(x_{k,t}-x_{k,t}^*)\nonumber\\
    & \geq & -(\ell+L_1+\alpha)\|x_{k,t}-x_{k,t}^*\|\|\bar x - x_{k,t}\| - \frac{\alpha}{2}\|x_{k,t}-x_{k,t}^*\|^2. \nonumber
\end{eqnarray}
Substituting the above two inequalities into \eqref{lm:aapp-3} and applying \eqref{lm:aapp-0} yields
\begin{eqnarray*}
    &&\hpk(\bar x) - \hpk(x_{k,t}^*) - \nabla \hpk(x_{k,t}^*)^{\!\top}\!(\bar x-x_{k,t}^*) - \frac{\alpha}{2}\|\bar x  - x_{k,t}^*\|^2\\
    &\geq& \hpk(\bar x) - \hpk(x_{k,t}) - \nabla\hpk(x_{k,t})^\top (\bar x - x_{k,t}) - \frac{\alpha}{2}\|\bar x-x_{k,t}\|^2\\
    && - \delta_x - (\ell+L_1+\alpha)\sqrt{\frac{2\delta_x}{\gamma+2\alpha}}\cdot\|\bar x - x_{k,t}\|.
\end{eqnarray*}
Further substituting this inequality to \eqref{lm:aapp-2} yields:
\begin{align}
    \label{lm:aapp-4}
    \hpk(\bar x) \geq &\,\,\hpk(x_{k,t}) \!+\! \gamma\|x_{k,t}\!-\!\txkm\|^2 - \gamma\|\bar x-\txkm\|^2 +\big(\gamma\!+\!\frac{\alpha}{2}\big)\|\bar x \!-\! x_{k,t}^*\|^2 \\
    & \,\,\, + \tilde{D}_k(\bar x, x_{k,t}) -3\delta_x - (2\ell+L_1+\alpha)\sqrt{\frac{2\delta_x}{\gamma+2\alpha}}\cdot\|\bar x - x_{k,t}\|.\nonumber
\end{align}
Noting that
\begin{align*}
    &\gamma\|x_{k,t} \!-\! \txkm\|^2 - \gamma\|\bar x \!-\! \txkm\|^2 = 2\gamma(x_{k,t}\!-\!\bar x)^{\top}(x_{k,t}\!-\!\txkm) - \gamma\|\bar x \!-\! x_{k,t}\|^2 \\
    & \qquad\qquad\qquad =  -2\gamma(\bar x\!-\! \txkm)^\top(x_{k,t}\!-\!\txkm) + 2\gamma\|x_{k,t}\!-\!\txkm\|^2 - \gamma\|\bar x \!-\! x_{k,t}\|^2,
\end{align*}
we have
\begin{eqnarray*}
    &&\gamma\|x_{k,t} - \txkm\|^2 - \gamma\|\bar x-\txkm\|^2  +\big(\gamma + \frac{\alpha}{2}\big) \|\bar x  -  x_{k,t}^*\|^2\\
    & = & -2\gamma(\bar x-\txkm)^\top(x_{k,t}-\txkm) + 2\gamma\|x_{k,t}-\txkm\|^2 \\
    & &  +\big(\gamma + \frac{\alpha}{2}\big) \|\bar x  -  x_{k,t}^*\|^2-\gamma\|\bar x - x_{k,t}\|^2\nonumber\\
    & = & -2\gamma(\bar x-\txkm)^\top(x_{k,t}-\txkm) + 2\gamma\|x_{k,t}-\txkm\|^2 + \frac{\alpha}{2}\|\bar x - x_{k,t}\|^2  \\
    && + \big(\gamma + \frac{\alpha}{2}\big)\|x_{k,t} - x_{k,t}^*\|^2 + 2\big(\gamma + \frac{\alpha}{2}\big) (\bar x - x_{k,t})^\top (x_{k,t} - x_{k,t}^*)\\
    & \geq & -2\gamma(\bar x-\txkm)^\top(x_{k,t}-\txkm) + 2\gamma\|x_{k,t}-\txkm\|^2 + \frac{\alpha}{4}\|\bar x - x_{k,t}\|^2 - 8\kappa_x\delta_x, 
\end{eqnarray*}
where the last inequality is due to  
inequality \eqref{lm:aapp-0},  $\kappa_x=\gamma/\alpha\geq1$, and 
$$(\bar x \!-\! x_{k,t})^{\!\top}\! (x_{k,t} \!-\! x_{k,t}^*) \!\geq\! - \frac{\alpha\|\bar x - x_{k,t}\|^2}{4(2\gamma+\alpha)} \!-\! \frac{2\gamma+\alpha}{\alpha}\|x_{k,t}\!-\!x_{k,t}^*\|^2.$$
Substituting the above result to \eqref{lm:aapp-4} yields 
\begin{eqnarray*}
    \hpk(\bar x) &\geq & \hpk(x_{k,t}) \!-\! 2\gamma(\bar x \!-\!\txkm)^{\!\top}\!(x_{k,t} \!-\! \txkm) \!+\! 2\gamma\|x_{k,t}\!-\!\txkm\|^2\nonumber\\
    && + \frac{\alpha}{4}\|\bar x \!-\! x_{k,t}\|^2 \!+\! \tilde{D}_k(\bar x, x_{k,t}) \!-\! 11\kappa_x\delta_x \!-\! (\!2\ell+L_1\!+\!\alpha)\sqrt{\frac{2\delta_x}{\gamma\!+\!2\alpha}}\cdot\|\bar x \!-\! x_{k,t}\|,
\end{eqnarray*}
which completes the proof. 
\end{proof}

\section{Proof of Lemma \ref{lemma:aapp}}\label{proof:lemma:aapp}
\begin{proof}
By the assumption that $\tilde{D}_k(x_{k,s-1},x_{k,s})\geq0,$ $\tilde{D}_{k}(\bar x, x_{k,s})\geq0$, $1\leq s\leq t$ and the definition of the positive constant $d$,  Lemma \ref{lemma:aapp-intermediate} indicates that 
\begin{align}
    \label{lm:aapp-5}
    \hpk(x) &\geq\,  \hpk(x_{k,s}) - 2\gamma(x -\tilde{x}_{k,s-1})^{\top}(x_{k,s} - \tilde{x}_{k,s-1}) \\
    &\quad\,\,+ 2\gamma\|x_{k,s}-\tilde{x}_{k,s-1}\|^2 + \frac{\alpha}{4}\|x - x_{k,s}\|^2 - \chi_0\nonumber
\end{align} 
where $x \in \{x_{k,s-1}, \bar x\}$ and $\chi_0 :=11\kappa_x\delta_x+(2\ell+L_1+\alpha)\sqrt{\frac{2\delta_x}{\gamma+2\alpha}}\cdot d$. As a result,
\begin{eqnarray}
    & & \frac{1}{2\sqrt{\kappa_x}}\hpk(\bar x) + \Big(1-\frac{1}{2\sqrt{\kappa_x}}\Big)\hpk(x_{k,s-1})\nonumber\\
    & \geq & \frac{1}{2\sqrt{\kappa_x}}\Big(\!\hpk(x_{k,s}) \!-\! 2\gamma(\bar x \!-\!\tilde{x}_{k,s-1})^{\!\top}\!(x_{k,s} \!-\! \tilde{x}_{k,s-1}) \!+\! 2\gamma\|x_{k,s}\!-\!\tilde{x}_{k,s-1}\|^2 \!+\! \frac{\alpha}{4}\|\bar{x} \!-\! x_{k,s}\|^2\Big)\nonumber\\ &&+\Big(\!1\!-\!\frac{1}{2\sqrt{\kappa_x}}\Big)\!\Big(\hpk(x_{k,s}) \!-\! 2\gamma(x_{k,s-1} \!-\! \tilde{x}_{k,s-1})^{\top}(x_{k,s} \!-\! \tilde{x}_{k,s-1}) \!+\! 2\gamma\|x_{k,s}\!-\!\tilde{x}_{k,s-1}\|^2\Big)\nonumber\\
    & & +\frac{\alpha}{4}\Big(1-\frac{1}{2\sqrt{\kappa_x}}\Big)\|x_{k,s-1}-x_{k,s}\|^2  - \chi_0  \nonumber\\
    & = & \hpk(x_{k,s}) - 2\gamma\!\cdot\!\left(\frac{\bar x}{2\sqrt{\kappa_x}} + \Big(1-\frac{1}{2\sqrt{\kappa_x}}\Big)\!\cdot\! x_{k,s-1} - \tilde{x}_{k,s-1}\right)^{\!\!\top}\!\!\big(x_{k,s}-\tilde{x}_{k,s-1}\big)      \nonumber\\
    & &  + 2\gamma\|x_{k,s}\!-\!\tilde{x}_{k,s-1}\|^2 +\frac{\alpha}{8\sqrt{\kappa_x}}\|\bar x -x_{k,s}\|^2 +\frac{\alpha}{4}\Big(1-\frac{1}{2\sqrt{\kappa_x}}\Big)\|x_{k,s-1}-x_{k,s}\|^2 - \chi_0 \nonumber.
\end{eqnarray}
Consequently, we have 
\begin{eqnarray}
    \label{lm:aapp-6}
    \hpk(x_{k,s})\!-\!\hpk(\bar x) & \!\leq\! & \Big(1\!-\!\frac{1}{2\sqrt{\kappa_x}}\Big)\!\left(\hpk(x_{k,s-1})\!-\!\hpk(\bar x)\right) \!+\! \chi_0 \!-\!2\gamma\|x_{k,s}\!-\!\tilde{x}_{k,s-1}\|^2 \nonumber\\
    & & + 2\gamma\!\cdot\!\left(\frac{\bar x}{2\sqrt{\kappa_x}} + \Big(1-\frac{1}{2\sqrt{\kappa_x}}\Big)\!\cdot\! x_{k,s-1} - \tilde{x}_{k,s-1}\right)^{\!\!\top}\!\!\big(x_{k,s}-\tilde{x}_{k,s-1}\big)\\
    & & -\frac{\alpha}{8\sqrt{\kappa_x}}\|\bar x - x_{k,s}\|^2 - \frac{\alpha}{8}\|x_{k,s}-x_{k,s-1}\|^2\nonumber.
\end{eqnarray}
Define $z_{k,s}:=\tilde{x}_{k,s} + 2\sqrt{\kappa_x}(\tilde{x}_{k,s} - x_{k,s})$. Note that $\tilde{x}_{k,s} = x_{k,s} + \frac{2\sqrt{\kappa_x}-1}{2\sqrt{\kappa_x}+1}(x_{k,s}-x_{k,s-1})$, we have the following recursive formula for $z_{k,s}$:
$$z_{k,s} = \Big(1-\frac{1}{2\sqrt{\kappa_x}}\Big)\cdot z_{k,s-1} + 2\sqrt{\kappa_x}\cdot(x_{k,s}-\tilde{x}_{k,s-1}) + \frac{1}{2\sqrt{\kappa_x}}\cdot\tilde{x}_{k,s-1}.$$
Consequently, we have 
\begin{eqnarray}
    \label{lm:aapp-7}
    \|z_{k,s}-\bar x\|^2 &=& \left\|\Big(1-\frac{1}{2\sqrt{\kappa_x}}\Big)z_{k,s-1} + \frac{\tilde{x}_{k,s-1}}{2\sqrt{\kappa_x}}-\bar x\right\|^2 + 4\kappa_x\left\|x_{k,s}-\tilde{x}_{k,s-1}\right\|^2\nonumber\\
    & & +4\sqrt{\kappa_x}\left(\Big(1-\frac{1}{2\sqrt{\kappa_x}}\Big)\cdot z_{k,s-1} + \frac{\tilde{x}_{k,s-1}}{2\sqrt{\kappa_x}}-\bar x\right)^{\!\!\top}\!\!\left(x_{k,s}-\tilde{x}_{k,s-1}\right).
\end{eqnarray} 
By definition of $z_{k,s-1}$, we have 
\begin{equation}
    \label{lm:aapp-8}
    (1-\frac{1}{2\sqrt{\kappa_x}})z_{k,s-1} + \frac{\tilde{x}_{k,s-1}}{2\sqrt{\kappa_x}} = 2\sqrt{\kappa_x}\tilde{x}_{k,s-1}-(2\sqrt{\kappa_x}-1)x_{k,s-1}.
\end{equation}
By the fact that $(a+b)^2\leq(1+c)a^2 + (1+c^{-1}) b^2$ for $\forall c>0$, we have
\begin{equation}
    \label{lm:aapp-9}
    \|\tilde{x}_{k,s-1}-\bar x\|^2\leq \frac{5}{4}\|x_{k,s}-\bar x\|^2 + 5\|x_{k,s}-\tilde{x}_{k,s-1}\|^2.
\end{equation}
Similarly, we also have 
\begin{eqnarray}
    \label{lm:aapp-10}
    & &\,\, \left\|\Big(1-\frac{1}{2\sqrt{\kappa_x}}\Big)z_{k,s-1} + \frac{\tilde{x}_{k,s-1}}{2\sqrt{\kappa_x}}-\bar x\right\|^2\\
    & \leq & \Big(\!1\!-\!\frac{1}{2\sqrt{\kappa_x}}\Big)^{\!2}\! \Big(1\!+\!\frac{5}{8\sqrt{\kappa_x}-5}\Big)\!\!\left\|z_{k,s-1}\!-\!\bar x\right\|^2 \!+\! \frac{1}{4\kappa_x}\Big(\!1\!+\!\frac{8\sqrt{\kappa_x}-5}{5}\Big)\!\!\left\|\tilde{x}_{k,s-1} \!-\! \bar x \right\|^2\nonumber\\
    &\leq & \Big(\!1\!-\!\frac{1}{6\sqrt{\kappa_x}}\Big)\!\!\left\|z_{k,s-1}\!-\!\bar x\right\|^2 \!+\! \frac{2\left\|\tilde{x}_{k,s-1} \!-\! \bar x \right\|^2}{5\sqrt{\kappa_x}}.   \nonumber
\end{eqnarray} 
Now, combining the inequalities \eqref{lm:aapp-7}-\eqref{lm:aapp-10} yields
\begin{eqnarray}
    \label{lm:aapp-11}
    \left\|z_{k,s}-\bar x\right\|^2 &\leq& \left(1-\frac{1}{6\sqrt{\kappa_x}}\right)\left\|z_{k,s-1}-\bar x\right\|^2 + \frac{\|x_{k,s}-\bar x\|^2}{2\sqrt{\kappa_x}} + 6\kappa_x\|x_{k,s}-\tilde{x}_{k,s-1}\|^2\,\,\qquad \nonumber\\
    & & + 8\kappa_x\left(\tilde{x}_{k,s-1}-\Big(1-\frac{1}{2\sqrt{\kappa_x}}\Big)x_{k,s-1}-\frac{\bar x}{2\sqrt{\kappa_x}}\right)^{\!\!\top}\!\left(x_{k,s}-\tilde{x}_{k,s-1}\right).
\end{eqnarray}
Combining \eqref{lm:aapp-11} with \eqref{lm:aapp-6} yields
\begin{eqnarray*}
    &&\hpk(x_{k,s})-\hpk(\bar x) + \frac{\alpha}{4}\|z_{k,s}-\bar x\|^2\\
    & \leq & \left(1-\frac{1}{6\sqrt{\kappa_x}}\right)\left(\hpk(x_{k,s-1})-\hpk(\bar x) + \frac{\alpha}{4}\|z_{k,s-1}-\bar x\|^2\right) + \chi_0.
\end{eqnarray*}
Consequently, by the fact that $z_{k,0} = x_{k,0}$ and $\chi = 6\sqrt{\kappa_x}\chi_0$, we prove this lemma. 
\end{proof}

\section{Proof of Lemma \ref{lemma:Flag-24}}\label{proof:lemma:Flag-24}
\begin{proof}
First of all, when Algorithm \ref{alg:main} executes Line 16, by Lemma \ref{lemma:distance}, we know
$\left\{x_{k,t}\right\}_{t=0}^{T_k}\cup\{w^x_k\}\subseteq \cX_k$. Next, let us prove this lemma holds for each Flag value. \vspace{0.3cm}\\
\textbf{Case Flag = 2.} In this situation,  $\hat{\phi}_{k(x_{k,T_k})}>\hat{\phi}_{k(x_{k,0})}+\chi+2\delta_y$ and the algorithm returns $w^x_k=x_{k,0}$. Now let us set $\bar x  = x_{k,0}$ and $d = \frac{\alpha}{2L_2}$ in Lemma \ref{lemma:aapp}. If  $\tilde{D}_k(x_{k,s-1},x_{k,s})\geq0$ and $\tilde{D}_k(\bar x, x_{k,s})\geq0$ hold for $1\leq s\leq T_k$, then Lemma \ref{lemma:aapp} immediately indicates that $\hpk(x_{k,T_k}) - \hpk(x_{k,0})  \leq   \chi$. This inequality further indicates 
$$\hat{\phi}_{k(x_{k,T_k})}-\hat{\phi}_{k(x_{k,0})}\leq \hpk(x_{k,T_k})-\hpk(x_{k,0}) + 2\delta_y\leq \cE + 2\delta_y,$$
which will cause a contradiction. As a result, there must be some $\tilde{D}_k(x_{k,s-1},x_{k,s})<0$ or $\tilde{D}_k(x_{k,0},x_{k,s})<0$.\vspace{0.3cm} \\
\textbf{Case Flag = 4.} In this case, set $\bar x  = w^x_k$ and $d = \frac{\alpha}{2L_2}$ in Lemma \ref{lemma:aapp}.  If  $\tilde{D}_k(x_{k,s-1},x_{k,s})\geq0$ and $\tilde{D}_k(\bar x, x_{k,s})\geq0$ hold for all $1\leq s\leq T_k$, then we have
\begin{align}
    \gamma\|w^x_k-x_{k,T_k}\|^2 &\leq  \hpk(x_{k,T_k}) - \hpk(w^x_k) + \delta_x \nonumber\\
    & \leq  \Big(1-\frac{1}{6\sqrt{\kappa_x}}\Big)^t\cdot\left(\hpk(x_{k,0})-\hpk(w^x_k) + \frac{\alpha}{4}\|w^x_k-x_{k,0}\|^2\right) + \chi + \delta_x\nonumber\\
    & \leq  \Big(1-\frac{1}{6\sqrt{\kappa_x}}\Big)^t\cdot\left(E_k + 2\delta_y\right) + \chi + \delta_x \nonumber\\
    & \leq   \Big(1-\frac{1}{6\sqrt{\kappa_x}}\Big)^t\cdot E_k + \chi + \delta_x + 2\delta_y\nonumber,
\end{align} 
which contradicts the necessary condition for returning $\text{Flag} = 4$. Therefore, similar to Case $\text{Flag}=2$,  there must exist some $\tilde{D}_k(x_{k,s-1},x_{k,s})<0$ or $\tilde{D}_k(x_{k,0},x_{k,s})<0$.
         
Next, let us prove that Algorithm \ref{alg:Exploit-Ncvx} must be able to find some $(u,v)$ through Line 3 to Line 10. According to our previous discussion, suppose $\tilde{D}_k(x_{k,s-1},x_{k,s})<0$ for some $1\leq s\leq T_k$. Therefore, we have \vspace{-0.2cm}
\begin{align}
    &\,\,\zeta\big(x_{k,s-1},x_{k,s}\big) \nonumber\\
    = &\,\,\hat{\phi}_{k (x_{k,s-1})\!}-\!\hat{\phi}_{k (x_{k,s})}\!-\!\hat{g}_{k(x_{k,s})}^{\!\top}\!(x_{k,s-1}\!-\!x_{k,s})\!-\!\frac{\alpha}{2}\|x_{k,s-1}\!-\!x_{k,s}\|^2\nonumber\\
    \leq & \,\,  \hpk(x_{k,s-1})\!-\!\hpk(x_{k,s}) + 2\delta_y  -\!  \nabla\hpk(x_{k,s})^{\!\top}\!(x_{k,s-1}\!-\!x_{k,s})     \!-\!\frac{\alpha}{2}\|x_{k,s-1}\!-\!x_{k,s}\|^2 \nonumber\\
    &\, + \left(\nabla\hpk(x_{k,s})-\hat{g}_{k(x_{k,s})}\right)^{\!\top}\!(x_{k,s-1}\!-\!x_{k,s})    \nonumber\\
    \leq & \,\, \tilde{D}_k(x_{k,s-1},x_{k,s})  - (\delta_x-2\delta_y)  -  \left(\sqrt{\frac{2\delta_x}{\gamma\!+\!\alpha}}\ell-\Delta_y\right)\cdot\|x_{k,s-1}\!-\!x_{k,s}\|    \nonumber\\
    < & \,\,   0 - 2\delta_y - \Delta_y\cdot\|x_{k,s-1}-x_{k,s}\| .    \nonumber
\end{align} 
Therefore, we have 
$$\tilde{D}_k(x_{k,s-1},\!x_{k,s})\!<\zeta\big(x_{k,s-1},x_{k,s}\big)<\!-\! 2\delta_y \!-\Delta_y\cdot\|x_{k,s-1}\!-\!x_{k,s}\|.$$
Similarly, we have 
$$\tilde{D}_k(x_{k,0},x_{k,s})\!<\!0\Longrightarrow\zeta\big(\!x_{k,0},\!x_{k,s}\!\big)\!<\!-\!2\delta_y\! -\Delta_y\cdot\|x_{k,0}-x_{k,s}\|.$$
Therefore, $(u,v)$ will be nonempty after running Line 3 to Line 10 of Algorithm \ref{alg:Exploit-Ncvx}. 
      
Finally, we prove that $(u,v)$ satisfies required properties. According to Algorithm \ref{alg:Exploit-Ncvx}, $u\in\left\{x_{k,s}\right\}_{s=0}^{T_k-1}\cup\left\{w^x_k\right\}$. Let $w$ be the associated $y$ variable of $u$ in the algorithm. By Lemma \ref{lemma:LeqFx0}, we have $\hat{\phi}_{k(u)}\leq  \hat{\phi}_{k(x_{k,0})}+\chi+\delta_x +4\delta_y$, which indicates 
$$\Phi(u)-\Phi(x_{k,0})\leq \hpk(u)-\Phi(x_{k,0})\leq \hat{\phi}_{k(u)} -\phi_{x_{k,0}}+2\delta_y\leq \cE+\delta_x+6\delta_y.$$
For ease of discussion, suppose $(u,v) \!=\! (x_{k,s-1},x_{k,s})$ for some $1 \!\leq\! s \!\leq\! T_k$. Then Line 5 of Algorithm \ref{alg:Exploit-Ncvx} indicates that $\zeta\big(\!x_{k,s-1},x_{k,s}\!\big)\!<\!-2\delta_y \!-\Delta_y\cdot\|x_{k,s-1}\!-\!x_{k,s}\|.$ Consequently, 
\begin{eqnarray}
    &&\!\hpk(x_{k,s-1})\!-\!\hpk(x_{k,s})  \! -\!  \nabla\hpk(x_{k,s})^{\!\top}\!(x_{k,s-1}\!-\!x_{k,s})     \!-\!\frac{\alpha}{2}\|x_{k,s-1}\!-\!x_{k,s}\|^2 \nonumber\\
    & \leq & \hat{\phi}_{k(x_{k,s-1})} - \hat{\phi}_{k(x_{k,s})} + 2\delta_y - \hat{g}_{k(x_{k,s})}^\top(x_{k,s-1}-x_{k,s}) \nonumber\\
    & & -\frac{\alpha}{2}\|x_{k,s-1}\!-\!x_{k,s}\|^2 - \left(\nabla\hpk(x_{k,s})-\hat{g}_{k(x_{k,s})}\right)^{\!\top}\!(x_{k,s-1}\!-\!x_{k,s})\nonumber\\
    & \leq & \zeta\big(\!x_{k,s-1},\!x_{k,s}\!\big) + 2\delta_y + \Delta_y\|x_{k,s-1}-x_{k,s}\|< 0.\nonumber
\end{eqnarray}
Similar argument applies to the case where $(u,v) = (x_{k,0},x_{k,s})$.  
\end{proof}

\section{Preliminary Numerical Experiments}  In this appendix, we present a preliminary experiment on a robust non-linear regression problem to validate the performance of our algorithm. Given the set of data points and the corresponding labels $\{(x_i,y_i)\}_{i=1}^N$, we aim to solve the NC-SC minimax problem \vspace{-0.1cm}
\begin{align}
    \min_{x\in\mathbb{R}^d} \!\!\!\max_{\,\,\,\,\,\,\{\delta_i\in\mathbb{R}^d\}_{i=1}^N}\,\, f\left(w,\{\delta_i\}_{i=1}^N\right) = \frac{1}{N}\sum_{i=1}^N\left(\phi\left((x_i+\delta_i)^\top w-y_i)\right)-\frac{\rho}{2}\|\delta_i\|^2\right)\nonumber \vspace{-0.1cm}
\end{align}
where $\phi(\theta)=\frac{\theta^2}{1+\theta^2}$ is the smooth biweight loss function \cite{beaton1974fitting}. Moreover, to encourage model robustness, an adversarial perturbation $\delta_i$ is added to each data point $x_i$ while a quadratic penalty term on $\delta_i$ controls the level perturbation, see e.g. \cite{lin2020gradient}. 

\begin{figure}[h]
\centering
\hspace{-0.05cm}\begin{subfigure}{.25\linewidth}
\includegraphics[width=\linewidth]{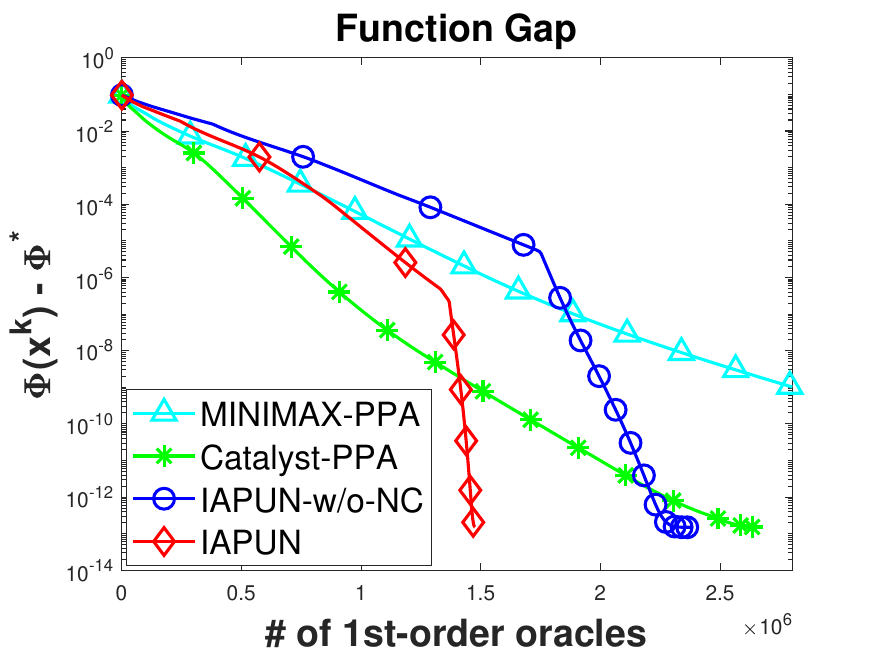}  
\end{subfigure}\hspace{-0.35cm}
\begin{subfigure}{.25\linewidth}
  \includegraphics[width=\linewidth]{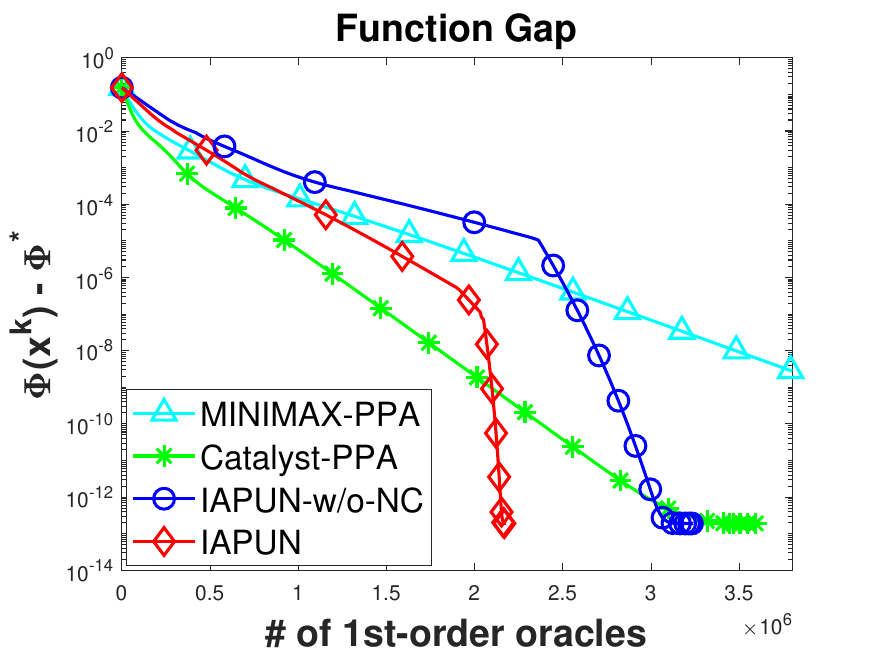} 
\end{subfigure}\hspace{-0.35cm}
\begin{subfigure}{.25\linewidth}
  \includegraphics[width=\linewidth]{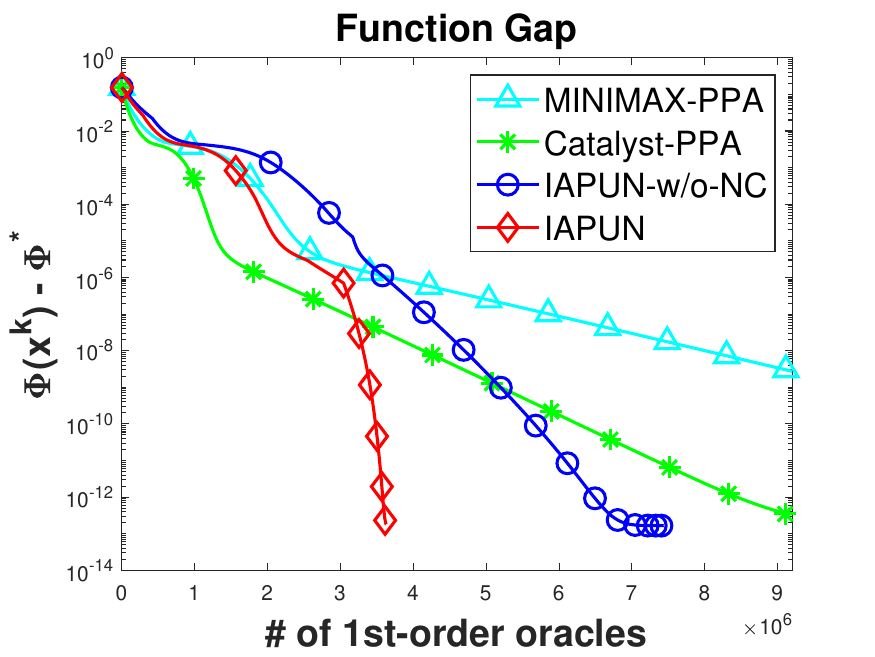} 
\end{subfigure}\hspace{-0.35cm}
\begin{subfigure}{.25\linewidth}
  \includegraphics[width=\linewidth]{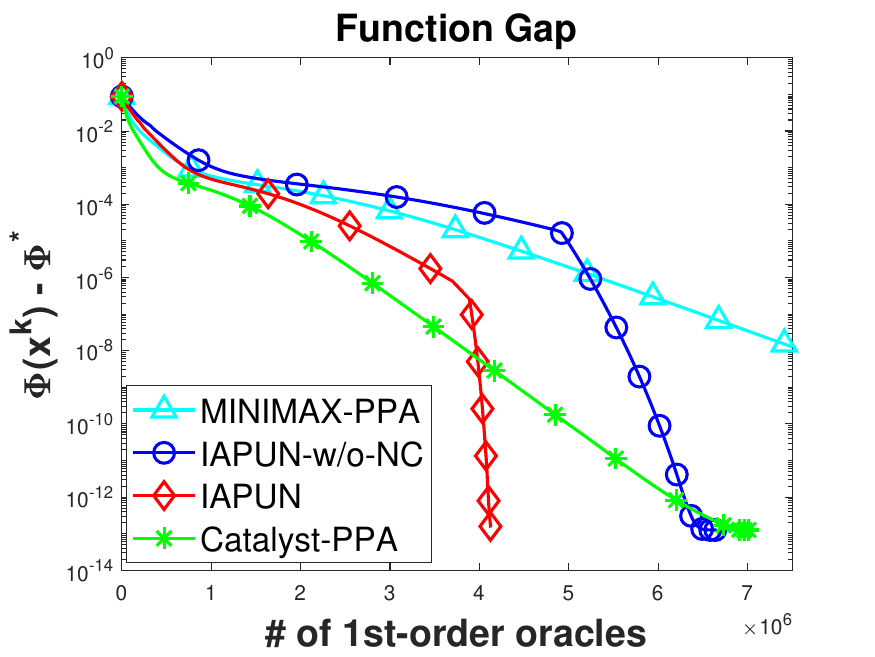}\!\!\!
\end{subfigure}

\vspace{0.2cm} % create some *vertical* separation between the graphs
\hspace{-0.05cm}\begin{subfigure}{.25\linewidth}
  \includegraphics[width=\linewidth]{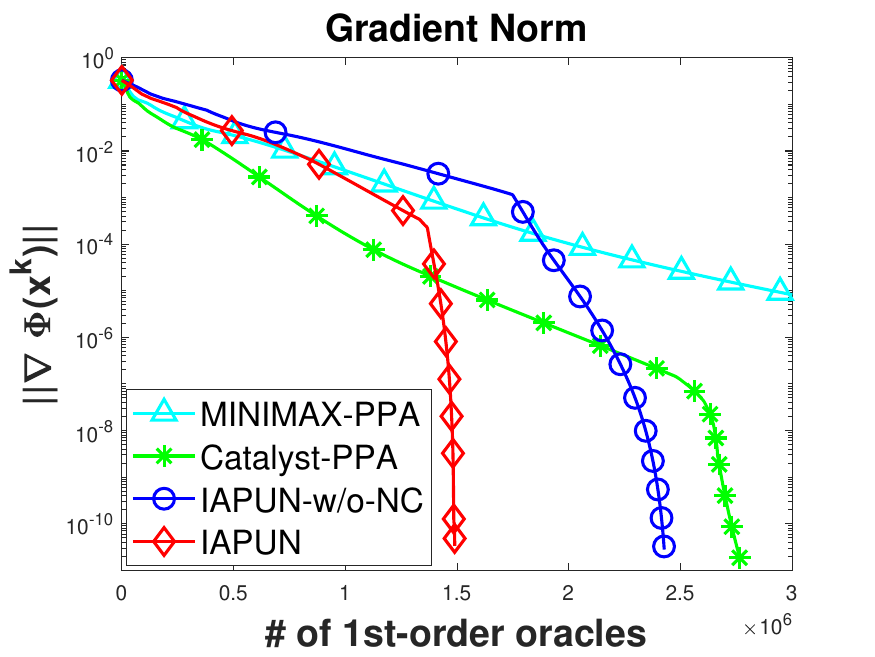} 
\end{subfigure}\hspace{-0.35cm}
\begin{subfigure}{.25\linewidth}
  \includegraphics[width=\linewidth]{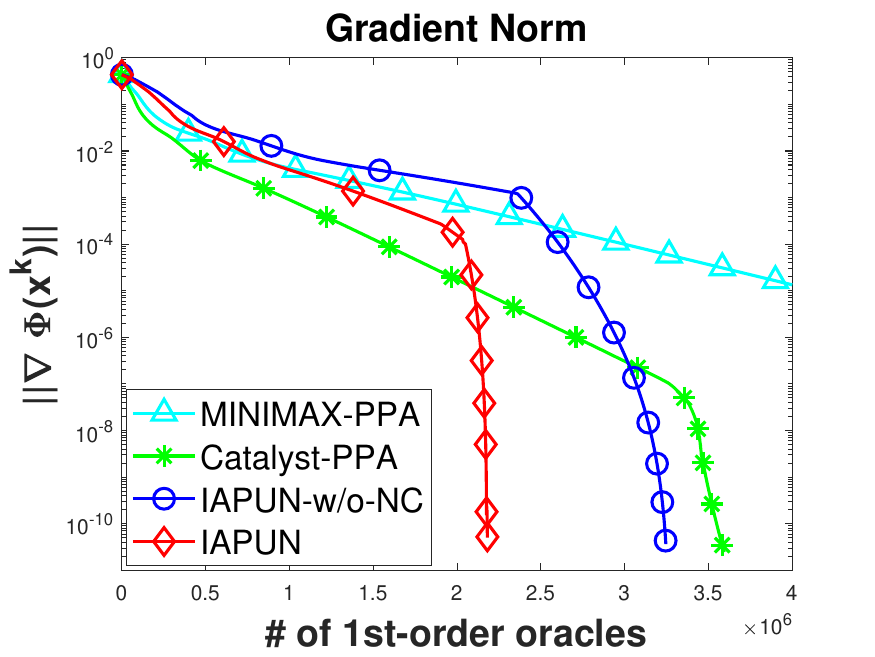}
\end{subfigure}\hspace{-0.35cm}
\begin{subfigure}{.25\linewidth}
  \includegraphics[width=\linewidth]{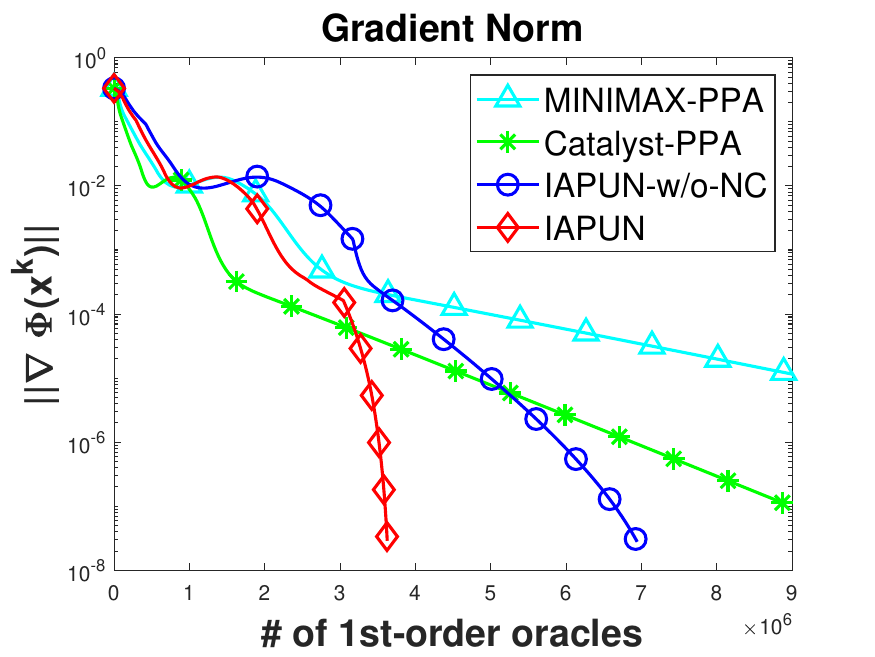}
\end{subfigure}\hspace{-0.35cm}
\begin{subfigure}{.25\linewidth}
  \includegraphics[width=\linewidth]{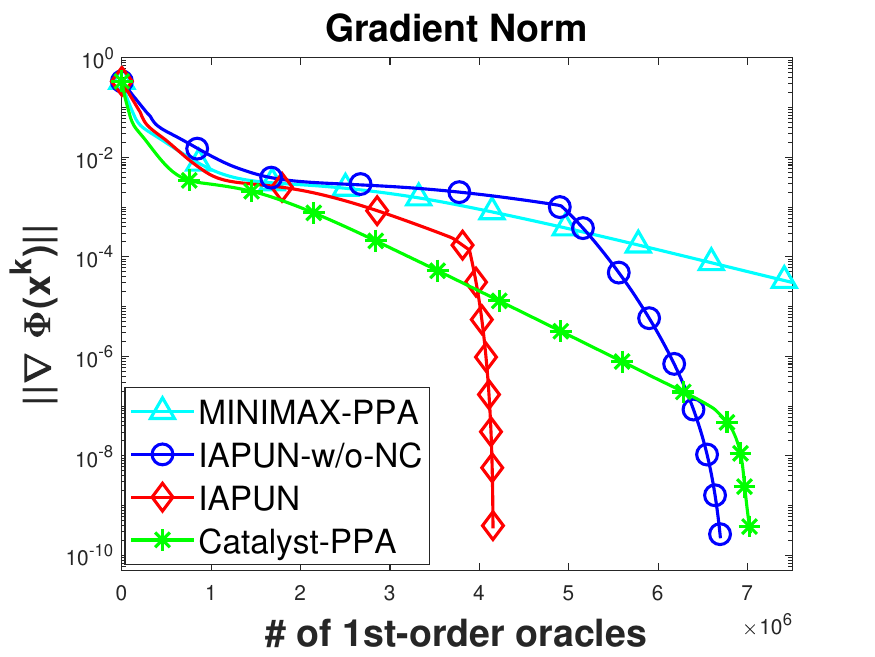}
\end{subfigure}\vspace{-0.05cm}
\caption{Experimental results for robust non-linear regression with adversarial perturbation for 4 randomly synthesized instances, each column contains the function gap and gradient norm of a randomly synthesized instance.  }
\label{fig:RobustReg}\vspace{-0.4cm}
\end{figure}

\noindent We compare our algorithm with the MINIMAX-PPA \cite{lin2020near} method, which is the optimal first-order algorithm that does not exploit higher-order smoothness. In addition, we also implement the Catalyst-PPA \cite{zhang2021complexity} as a second benchmark, which is a Catalyst acceleration of the MINIMAX-PPA algorithm that reduces a few poly-logarithmic factors in the complexity. Besides the IAPUN method, we also implement a variant of this method that does exploit NC-pairs. Following \cite{carmon2017convex}, we set $d=30$, $N=60$, and we generate $x_i\sim \mathcal{N}(0,I_d)$ and $y_i\sim\mathcal{N}(0,1)$. All methods are initialized at 0 and all the PPA subproblems are properly warm started. Numerical experiments under several randomly generated datasets are presented in Figure \ref{fig:RobustReg}.

In the experimental results, it can be observed that the IAPUN algorithm and its variant without NC-pair exploitation perform better in the high-accuracy regime while the MINIMAX-PPA and Catalyst-PPA methods are often more efficient in the low accuracy regime. \vspace{-0.3cm}

%yet the implementation and the parameter tuning are much more complicated for these (accelerated) PPA-based algorithms, careful subproblem solution accuracy and warm starting should be incoporated. }

%\end{appendices}

%\appendix
%\section{An example appendix} 
%\lipsum[71]

\bibliographystyle{alpha}
\bibliography{main}
\end{document}